\documentclass[11pt,reqno]{article}

\usepackage{amsmath}
\usepackage{amssymb}
\usepackage{amsthm}
\usepackage{graphicx}
\usepackage{tikz}
\usepackage{xcolor}
\usepackage{mathtools}
\usepackage{enumitem}
\usepackage{url}

\allowdisplaybreaks

\usepackage[margin=1.2in]{geometry}

\usepackage{bookmark}

\def\E{\mathbb{E}}
\def\var{\mathbb{Var}}

\def\dist{\mathrm{dist}}

\def\E{\mathbb{E}}
\def\R{\mathbb{R}}
\def\P{\mathbb{P}}
\def\Z{\mathbb{Z}}

\def\eps{\varepsilon}
\def\del{\delta}
\def\gam{\gamma}

\def\cB{\mathcal {B}}
\def\cC{\mathcal {C}}

\def\cG{\mathcal {G}}
\def\cP{\mathcal {P}}

\def\1{\mathbf{1}}

\def\lam {\lambda}

\def\tce{t_c + \eps}
\def\tce2{t_c + \frac{\eps}{2}}

\def\bT{\mathbb{T}}

\def\var{\text{var}}

\newtheorem*{theorem*}{Theorem}
\newtheorem{theorem}{Theorem}
\newtheorem{lemma}[theorem]{Lemma}
\newtheorem{cor}[theorem]{Corollary}

\newtheorem*{defn*}{Definition}
\newtheorem{assm}{Assumption}
\newtheorem{prop}[theorem]{Proposition}
\newtheorem*{prop*}{Proposition}
\newtheorem{conj}[theorem]{Conjecture}
\newtheorem*{conj*}{Conjecture}

\newtheorem*{fact*}{Fact}

\newtheorem{rmk}{Remark}

\newcommand{\ord}{\mathrm{ord}}
\newcommand{\dis}{\mathrm{dis}}
\newcommand{\err}{\mathrm{err}}
\newcommand{\bydef}{\coloneqq}
\newcommand{\Potts}{\mathrm{Potts}}
\newcommand{\free}{\mathrm{free}}
\newcommand{\wire}{\mathrm{wire}}

\newcommand{\bb}[1]{\mathbb{#1}}
\newcommand{\cc}[1]{\mathcal{#1}}

\renewcommand{\R}{\bb R}
\newcommand{\N}{\bb N}

\newcommand{\ob}[1]{\left(#1\right )} 
\newcommand{\cb}[1]{\left[#1\right ]} 
\newcommand{\abs}[1]{\left\vert#1\right\vert} 

\begin{document}
\title{Finite-size scaling, phase coexistence, and algorithms for the random cluster model on random graphs}

\author{Tyler Helmuth\thanks{Department of Mathematical Sciences, Durham University,  tyler.helmuth@durham.ac.uk.}
\and
Matthew Jenssen\thanks{School of Mathematics, University of Birmingham, m.jenssen@bham.ac.uk.}
\and
Will Perkins\thanks{Department of Mathematics, Statistics, and Computer Science, University of Illinois at Chicago, math@willperkins.org.  Supported in part by NSF grants DMS-1847451 and CCF-1934915.}}

\date{\today}

\maketitle
\thispagestyle{empty}

\begin{abstract}
  For $\Delta \ge 5$ and $q$ large as a function of $\Delta$, we give
  a detailed picture of the phase transition of the random cluster
  model on random $\Delta$-regular graphs.  In particular, we
  determine the limiting distribution of the weights of the ordered
  and disordered phases at criticality and prove exponential decay of
  correlations and central limit theorems away from criticality.
  Our techniques are based on using polymer models and the cluster
  expansion to control deviations from the ordered and disordered
  ground states.  These techniques also yield efficient approximate
  counting and sampling algorithms for the Potts and random cluster
  models on random $\Delta$-regular graphs at \emph{all} temperatures
  when $q$ is large. This includes the critical temperature at which
  it is known the Glauber and Swendsen-Wang dynamics for the Potts
  model mix slowly.  We further prove new slow-mixing results for
  Markov chains, most notably that the Swendsen-Wang dynamics mix
  exponentially slowly throughout an open interval containing the
  critical temperature. This was previously only known at the critical
  temperature.
  
  Many of our results apply more generally to $\Delta$-regular graphs
  satisfying a small-set expansion condition.
\end{abstract}

\section{Introduction}

The random cluster model on a graph $G = (V,E)$ with parameters $q,
\beta\geq 0$ is the measure $\mu_G$ on $\{0,1\}^E$ with
\begin{equation}
  \label{eq:PF}
\mu_G(A) \bydef \frac{q^{c(A) } (e^{\beta}-1)^{|A|}}{ Z_G(q,\beta)}\, ,
\qquad
Z_G(q,\beta) \bydef \sum_{A \subseteq E} q^{c(A)} (e^{\beta}-1)^{| A|}\,  ,
\end{equation}
where $c(A)$ is the number of connected components of $(V,A)$.
Setting $p \bydef 1- e^{-\beta}$ gives a description of $\mu_G$ as a
tilted bond percolation model with edge probability $p\in \cb{0,1}$:
\begin{equation*}
  \label{eq:PFp}
  Z_G(q,- \log (1-p)) = \sum_{A \subseteq E} q^{c(A)} \left  (\frac{p}{1-p} \right)^{| A|} =
  \frac{1}{(1-p)^{|E|}} \sum_{A\subseteq E}q^{c(A)}p^{|A|}(1-p)^{|E\setminus A|}\, .
\end{equation*}

The random cluster model is a generalization of the \emph{$q$-color ferromagnetic
  Potts model}, which, for $q$ a positive integer, is the probability
distribution on $[q]^{V}$  defined by
\begin{equation*}
\mu_G^{\Potts}(\sigma) \bydef  \frac{1}{Z^{\Potts}_G(q,\beta)}\prod_{\{u,v\}\in E } e^{\beta \mathbf 1
  _{\sigma_u=\sigma_v}} \, , \qquad
Z^{\Potts}_G(q,\beta) \bydef \sum_{\sigma \in [q]^V} \prod_{\{u,v\} \in E } e^{\beta \mathbf 1 _{\sigma_u=\sigma_v}}  \,.
\end{equation*}
In particular, the case $q=2$ case is the Ising model.  The connection
between the Potts and random cluster models is that for integer $q$,
\begin{equation*}
Z_G(q,\beta) = Z^{\Potts}_G(q,\beta)\, ,
\end{equation*}
and moreover, there is a natural coupling of the measures $\mu_G$ and
$\mu_G^{\Potts}$. The coupling is as follows. Given an edge set $A$
distributed according to the random cluster measure $\mu_G$,
independently assign a uniformly chosen color from $[q]$ to each connected component of $(V,A)$ to obtain a coloring
$\sigma\in [q]^{V}$.  The distribution of $\sigma$ is
$\mu_G^{\Potts}$~\cite{edwards1988generalization}.  For an introduction
to the random cluster model, see~\cite{grimmett2006random}.

This paper concerns a relatively complete set of results about the
probabilistic and algorithmic behavior of the large-$q$ random cluster
model on random $\Delta$-regular graphs.  In particular, we obtain a
detailed description of the phase diagram; establish strong
correlation decay and finite-size scaling statements; prove central
limit theorems off criticality; obtain efficient approximate counting
and sampling algorithms at \emph{all} temperatures $\beta>0$; and
establish slow-mixing of standard Markov chains in a neighborhood of
the critical temperature $\beta_{c}$. Many of our results apply more
generally to $\Delta$-regular graphs satisfying a small-set expansion
condition, see Section~\ref{sec:expansion-profiles}.  We will shortly
give precise statements of these results, but before doing this we
briefly give some context and an outline of our methods.

Our techniques are based on polymer models and the cluster expansion,
tools developed to investigate the phase diagrams of statistical
physics models on
lattices~\cite{gruber1971general,pirogov1975phase,kotecky1986cluster,laanait1991interfaces}. In
particular, we adapt to the random graph and expander setting the idea
from~\cite{laanait1991interfaces} of analyzing the Potts model phase
transition by controlling the ordered and disordered phases of the
random cluster model via separate convergent cluster expansions. The
key to this approach is obtaining convergent ordered and disordered
expansions for parameter regimes that overlap --- in particular, the
expansions both converge at the critical temperature. These expansions
give us strong control on the dominant and sub-dominant contributions
to the partition function, and enable us to prove our probabilistic
and algorithmic results. While the use of expansion methods to obtain
probabilistic results is well-known, algorithmic implications are more
recent~\cite{HelmuthAlgorithmic2,JKP2,liao2019counting,galanis2021fast,chen2021sampling}.

The most crucial technical aspect of this paper is thus the
development of convergent expansions, and our main innovation here is
a 
polymer model that applies to the ordered phase on expander
graphs. This relies on an inductive construction of polymers that
circumvents the difficulty created by the non-local weight $q^{c(A)}$
present in the random cluster model. In prior work studying the random
cluster model on $\mathbb{Z}^{d}$ via Pirogov--Sinai theory, this
non-local weight was handled by using notions of boundaries arising
from the topology of Euclidean space. Our inductive construction
defines a notion of `boundary' that encodes the connected components
of $A$, and hence the computation of $q^{c(A)}$, for typical edge sets
$A$. The success of this encoding, and its usefulness for deriving a
convergent expansion, relies crucially on (i) expansion properties of
the underlying graph and (ii) the fact that in the ordered phase,
typical edge sets $A$ consist of a large fraction of all edges. The
second point is one aspect of the fact that the phase transition of
the random cluster model is first order when $q$ is large. This leads
to the resulting boundaries being geometrically small, which is
important for obtaining a convergent expansion.

The techniques typically used to understand statistical physics models
on random graphs are very different from ours: typical methods include
the first and second moment methods, the cavity method, and the
interpolation
method~\cite{mezard2009information,montanari2012weak,sly2014counting,dembo2014replica,coja2018charting,coja2020replica}.
Using polymer models and the cluster expansion allows us to obtain
results that are not, to date, accessible via the aforementioned
techniques.  The strengths of our approach include: the ability to
make statements about every vertex or pair of vertices in a graph
(e.g., Theorem~\ref{thmPhases}, part (6)); a precise characterization
of phase coexistence (Theorem~\ref{thmCritical}); and control of both
the ordered and disordered contributions to the partition function at
and away from criticality which leads to strong algorithmic
consequences (Theorems~\ref{thmAlgorithm} and~\ref{thmSlowMix}).  On
the other hand, our approach is inherently perturbative in that it
requires $q$ to be large, and it does not as readily yield explicit
formulae for critical thresholds.

\subsection{The phase diagram of the random cluster model on random graphs}
\label{secProbabilistic}

On $\Z^{d}$, meaning on sequences of graphs $G_{n}\uparrow \Z^{d}$ in
an appropriate sense, a great deal is known about the random cluster
model, see~\cite{duminilcopin} and references therein. When $q\geq 1$
these models are known to undergo a phase transition at a critical
temperature $\beta_{c}(q)$ from a \emph{disordered state
  ($\beta<\beta_{c}$)} to an \emph{ordered state ($\beta>\beta_{c}$)}.
The nature of this transition depends on the value of $q$ and the
dimension $d$. For the present paper the most relevant results concern
when $q$ is large. In this case configurations in the disordered state
typically consist of relatively few edges, while in the ordered state
typical configurations have relatively few missing edges. Moreover,
exactly at $\beta_{c}$ typical configurations look like either an
ordered or a disordered configuration. In physical parlance, the phase
transition is \emph{first-order} ~\cite{laanait1991interfaces}. Finer
results concerning finite-size scaling are also
known~\cite{borgs1991finite}. Roughly speaking, these results concern
how $\abs{V(G_{n})}^{-1}\log Z_{G_{n}}$ differs from
$\lim_{n\to\infty} \abs{V(G_{n})}^{-1}\log Z_{G_{n}}$, i.e., the
corrections to the leading order behavior of $\log Z_{G_{n}}$ as
$G_{n}\uparrow \Z^{d}$.  Below we will identify a first-order phase
transition for the random cluster model on random regular graphs
(Theorem~\ref{thmPhases}) and determine the finite-size scaling of the
(random) log partition function (Theorem~\ref{thm:Q}).

\subsubsection{Expansion profiles}
\label{sec:expansion-profiles}

To state the class of graphs to which our results apply, we need a
refined notion of edge expansion.  The \emph{expansion profile} of a
$\Delta$-regular graph $G=(V,E)$ is
\begin{equation} 
  \phi_G(\alpha) \bydef \min_{S \subset V, |S| \le \alpha |V|} \frac{ |
    E(S,S^c)|  }{  \Delta |S|  }, \qquad \alpha\in(0,1/2], 
\end{equation}
where $E(S,S^{c})\subset E$ is the set of edges with one vertex in
$S$, one in $S^{c}$. For $\Delta \in \{3,4,\dots\}$ and
$\delta\in (0,1/2)$ we will be interested in the family
$\cG_{\Delta, \del}$ of $\Delta$-regular graphs that satisfy:
\begin{enumerate}
\item $\phi_G(1/2) \ge 1/10$,
\item $\phi_G(\del) \ge5/9 $.
\end{enumerate}
We note that the constants $1/10$ and  $5/9$ are somewhat arbitrary; what we
use in our proofs is that they are greater than $0$ and $1/2$, respectively.

\subsubsection{Locally tree-like graphs and local convergence of
  probability measures}
\label{sec:locally-tree-like}
Given a graph $G$, 
let $B_T(v)$ denote the depth-$T$ neighborhood of a vertex
$v \in V(G)$.  A sequence of graphs $G_n$ is \emph{locally tree-like}
if for every $T>0$, with probability tending to one as $n\to \infty$ over the choice of
a uniformly random vertex $v$ from $G_n$, $B_T(v)$ is a tree.

Recall that \emph{random cluster measures on the infinite
  $\Delta$-regular tree} $\bT_\Delta$ can be defined by taking weak
limits of measures on finite trees with boundary conditions.  Two
random cluster measures $\mu^{\free}$ and $\mu^{\wire}$ on
$\bT_{\Delta}$ are of particular importance: they are respectively
obtained by taking weak limits with free boundary conditions and with
wired boundary conditions, i.e., all leaves `wired' into one connected
component. See~\cite[Chapter 10]{grimmett2006random} for more details,
including a proof that these weak limits exist and are unique.
The measure $\mu^{\free}$ is particularly simple as it is an
independent edge percolation measure.

In what follows we adopt the convention that the index of a graph
sequence denotes the number of vertices in the graph:
$|V(G_{n})|=n$. We further assume $n$ is increasing, but not
necessarily through consecutive integers, so as to ensure $n\Delta/2$
is an integer. Limits as $n\to \infty$ are understood in this sense.
  
For a sequence of $\Delta$-regular graphs $G_{n}$ we say a corresponding sequence
$\mu_n$ of probability measures on $\{0,1\}^{E(G_{n})}$
\emph{converges locally} to a random cluster measure $\mu_{\infty}$ on
the infinite $\Delta$-regular tree $\bT_{\Delta}$, denoted
$\mu_n \xrightarrow{loc} \mu_{\infty}$, if for every $\eps, T>0$ and
$n$ sufficiently large, with probability at least $1-\eps$ over the
choice of a random vertex $v$ from $G_n$, the distribution of $\mu_n$
restricted to $B_T(v)$ is within $\eps$ total variation distance of
the distribution of $\mu_{\infty}$ restricted to the depth-$T$
neighborhood of the root.
See~\cite{montanari2012weak,sly2014counting} for examples and more
details of this notion of convergence.

\subsubsection{Main probabilistic results and related literature}
\label{sec:results}

The statements of our results require some notation.  Fix
$\Delta, \del$ in the definition of $G_{\Delta,\del}$ and let
$\eta \bydef \min \{1/100, \del/5 \}$.  Given a graph $G_n=(V,E)$ on
$n$ vertices, let $\Omega_n = \{0,1\}^{E}$ and let
\begin{align*}
\Omega_{\dis} &\bydef \{ A \in \Omega_n : |A| \le \eta |E| \} \,  \\
\Omega_{\ord} &\bydef \{ A \in \Omega_n : |A| \ge (1-\eta) |E| \} \\
\Omega_{\err} &\bydef \Omega_n \setminus (\Omega_{\dis} \cup \Omega_{\ord}) \,
\end{align*}
so that
$\Omega_n = \Omega_{\dis} \sqcup \Omega_{\ord} \sqcup \Omega_{\err}$.
We write $\mu_n$ for the random cluster measure on $G_n$.

Recall that a sequence of probability measures $\mu_{n}$ on $\Omega_n$
has \emph{exponential decay of correlations with rate $\epsilon>0$} if
there exists a $C>0$ such that
\begin{equation*}
  \abs{\mu_{n}(e,f)-\mu_{n}(e)\mu_{n}(f)} \leq Ce^{-\epsilon \cdot
    \mathrm{dist}_{G_{n}}(e,f)}, \qquad \text{for all } e,f\in E(G_{n}),
\end{equation*}
where $\mathrm{dist}_{G_{n}}(\cdot,\cdot)$ is the graph distance on
$G_{n}$ and we have used the customary abuse of notation
$\mu_{n}(e) = \mu_{n}( \{ A\subset E : e\in A\})$ and similarly for
other marginals.

Finally with $\mathbf A $ denoting the random edge subset drawn according to a random cluster measure, given a sequence of random cluster measures $\mu_n$ we say $|\mathbf A|$ obeys a  central limit theorem under $\mu_n$ if for each  $t \in \R$,
\begin{align*}
\lim_{n \to \infty} \mathbb P \left( \frac{|\mathbf A |- \E_{\mu_n} |\mathbf A|}{  \sqrt{\var_{\mu_n} (|\mathbf A|)}  }  \le t \right) = \frac{1}{\sqrt{2 \pi}} \int_{- \infty}^{t} e^{-x^2/2} \, dx \,.
\end{align*}

\begin{theorem}
  \label{thmPhases}
  Suppose $\Delta\geq 5, \del >0$. For $q=q(\Delta,\del)$ sufficiently
  large there exists a $\beta_{c}(q,\Delta)$
  so that the following holds for every sequence
  $G_{n}\in \cc G_{\Delta,\delta}$ of locally tree-like graphs.
  \begin{enumerate}
  \item The limit $\lim_{n\to\infty}n^{-1}\log Z_{G_{n}}$ exists and  is an analytic function of
    $\beta$ on $(0,\infty)\setminus \{\beta_{c}\}$.  For $\beta \le
    \beta_c$ the limit equals $f_{\dis}=\log q + \frac{\Delta}{2}
    \log\left(1+\frac{e^{\beta}-1}{q}\right)$, while for $\beta \ge
    \beta_c$ the limiting value is given by a function
    $f_{\ord}(q,\Delta,\beta)$ defined in
    Section~\ref{sec:phase-transition}. 
  \item For $\beta < \beta_c$,
    $\limsup_{n \to \infty} n^{-1} \log \mu_n(\Omega_n \setminus \Omega_{\dis}) <0$.
  \item For $\beta > \beta_c$,
    $\limsup_{n \to \infty} n^{-1} \log \mu_n(\Omega_n \setminus \Omega_{\ord}) <0$.
  \item For $\beta < \beta_c$, $\mu_{n} \xrightarrow{loc} \mu^{\free}$ as $n\to \infty$.
  \item For $\beta > \beta_c$, $\mu_{n} \xrightarrow{loc} \mu^{\wire}$ as $n\to \infty$.
\item For $\beta\neq\beta_{c}$, $\mu_{n}$ exhibits exponential decay
  of correlations.
  \item For $\beta \neq \beta_c$, $|\mathbf A|$ obeys a  central limit theorem under $\mu_n$. 
  \item For all $\beta\ge0$, $ \mu_n(\Omega_{\err}) = O( e^{-n})$. \label{itemslow}
  \end{enumerate}
\end{theorem}

Theorem~\ref{thmPhases} gives a rather complete description of the
phase transition and probabilistic properties of the random cluster model on locally tree-like
graphs in $\cG_{\Delta,\del}$.  At all temperatures, all but an
exponentially small fraction of the measure is on configurations with
at most an $\eta$- or at least a $(1-\eta)$-fraction of all edges.
There is a unique phase transition at $\beta_c$, $\Omega_{\dis}$ has
all but an exponentially small fraction of the measure for
$\beta<\beta_c$, and $\Omega_{\ord}$ has all but an exponentially
small fraction of the measure for $\beta>\beta_c$. Correlations decay
exponentially at $\beta\neq\beta_{c}$. In fact, as the proof will
show, we can make stronger statements about correlation decay
conditional on $\Omega_{\dis}$ or $\Omega_{\ord}$; see also
Lemma~\ref{lemPolyExpDecay} below.  As part of our proof we
determine the critical point asymptotically in $q$:
$\beta_c(q,\Delta)= (1+o_q(1))\frac{2 \log q}{\Delta}$, see
Section~\ref{sec:phase-transition}.

For the $q$-color Potts model, some of these results were known
previously, and without the restriction that $q$ is large.  Dembo,
Montanari, and Sun~\cite{dembo2014replica} proved that for \emph{any}
sequence $G_n$ of $2\Delta$-regular locally tree-like graphs the limit
$\frac{1}{n} \log Z_{G_n}$ exists for all $\beta$ and equals the
replica-symmetric Bethe formula for the free energy, given implicitly
by a variational formula.  For random $\Delta$-regular graphs Galanis,
{\v{S}}tefankovi{\v{c}}, Vigoda, and
Yang~\cite{galanis2016ferromagnetic} established a detailed picture of
the phase transition, and they determined $\beta_c$ explicitly:
\begin{equation}
\label{eqBetaCexplicit}
\beta_c (q, \Delta) =   \log  \frac{  q-2 }{ (q-1)^{1-2/\Delta} -1   } \,
\end{equation}
for $q\ge 3 $ and $\Delta \ge 3$ integers.  They also proved versions
of Theorem~\ref{thmPhases} parts (2), (3), and (7) for the Potts model
with slightly different definitions of $\Omega_{\ord}$ and
$\Omega_{\dis}$.\footnote{The definitions of
  $\Omega_{\ord}$ and $\Omega_{\dis}$
  in~\cite{galanis2016ferromagnetic} specify the number of vertices
  receiving each of the $q$-colors.}  Taken together with the results
of~\cite{dembo2014replica} this implies the formula~\eqref{eqBetaCexplicit} for
$\beta_c$ for integral $q\geq 3$ holds for all sequences of $\Delta$-regular
locally tree-like graphs.  For the case $q=2$ of the Ising model, Dembo
and Montanari~\cite{dembo2010ising} and Montanari, Mossel, and
Sly~\cite{montanari2012weak} proved local convergence results for
locally tree-like graphs. See also~\cite{sly2014counting} for further
results on general $2$-spin models on locally tree-like graphs.  Giardin{\`a}, Giberti, van der Hofstad, and Prioriello~\cite{giardina2015quenched} proved a central limit theorem for the magnetization of Ising model in random regular graphs in the uniqueness regime. To the best of our knowledge Theorem~\ref{thmPhases} gives the first central limit theorem in a supercritical phase of a spin model on a sparse random graph. 

It is important to note that $\beta_c$ in Theorem~\ref{thmPhases} is
\emph{not} the Gibbs uniqueness threshold $\beta_u$ of the random
cluster model on the infinite tree. Instead it is the `order-disorder
threshold' in the terminology of~\cite{galanis2016ferromagnetic}.  In
particular, $\beta_u < \beta_c$. We also remark that there is another
uniqueness threshold $\beta_{u}^{\star}>\beta_{c}$ conjectured by
H{\"a}ggstr{\"o}m~\cite{haggstrom1996random}, but this conjecture
concerns a class of Gibbs measures that does not include
$\mu^{\free}$.
 
Note that the correlation decay given in part (6) of
Theorem~\ref{thmPhases} is very strong compared to decorrelation
statements for random graphs outside the range of tree uniqueness
obtained by other methods.  The statements in,
e.g.~\cite{coja2018charting,coja2020replica}, assert that the
correlation between two randomly chosen vertices in the graph tends to
$0$ with high probability, while the correlation decay property in
part (6) holds for \emph{every} pair of vertices, and moreover the
decay is exponential in the distance.

 Next we turn our attention more specifically to random
 $\Delta$-regular graphs, i.e., when $G_n$ is chosen uniformly at
 random from the set of all $\Delta$-regular graphs on $n$
 vertices. Recall we implicitly assume $n \Delta/2$ is an integer. As
 alluded to above, Theorems~\ref{thmPhases} and~\ref{thmSlowMix} apply
 to the random $\Delta$-regular graph when $\Delta\ge 5$: in
 Proposition~\ref{prop:RRG-exp} below we cite results showing that
 there is a $\del>0$ such that realizations of the random graph belong
 to $\cG_{\Delta,\del}$ with high probability.  Here and in what
 follows, we say a property $P_{n}$ of graphs on $n$ vertices holds
 with high probability if $\P\cb{ P_{n}} = 1-o(1)$ as $n \to \infty$.

Theorem~\ref{thmPhases} primarily concerned the behavior at
$\beta\neq\beta_{c}$, but it is also interesting to investigate the
behavior precisely at
$\beta_{c}$.
For the Potts model, Galanis, {\v{S}}tefankovi{\v{c}}, Vigoda, and
Yang~\cite{galanis2016ferromagnetic} showed that with high probability
over the choice of $G_n$, at $\beta=\beta_c$ both
$\mu_n(\Omega_{\ord}) \ge n^{-c}$ and
$\mu_n(\Omega_{\dis}) \ge n^{-c}$ for some constant $c>0$, and also
that $\mu_n(\Omega \setminus (\Omega_{\dis} \cup \Omega_{\ord}))$ is
exponentially small.\footnote{In fact~\cite{galanis2016ferromagnetic}
  uses slightly different definitions of $\Omega_{\ord}$ and
  $\Omega_{\dis}$, as was mentioned earlier.}  This is a
logarithmic-scale phase coexistence result, and using this, 
they proved that the Swendsen--Wang dynamics mix slowly at
criticality.  Our next result gives a complete phase coexistence result for the random
cluster model, and hence also the Potts model, on the random
$\Delta$-regular graph for $\Delta \ge 5$ and $q$ large by determining
precisely the limiting  distribution of $\mu_n(\Omega_{\dis}) $ and
$\mu_n(\Omega_{\ord}) $ at criticality. 
 \begin{theorem}
 \label{thmCritical}
For $\Delta \ge 5$ and $q =q(\Delta)$ large enough, there
 is a non-constant, positive random variable $Q$ so that for the random cluster model on the random $\Delta$-regular graph at $\beta = \beta_c$:
 \begin{enumerate}
\item The random variable $\mu_n(\Omega_{\dis})$ converges in distribution to $1/(Q+1)$ and $\mu_n(\Omega_{\ord})$ converges in distribution to $Q/(Q+1)$ as $n \to \infty$
\item  The random cluster measure $\mu_n$ conditioned on $\Omega_{\dis}$ and on $\Omega_{\ord}$ converges locally                       
to $\mu^{\free}$ and $\mu^{\wire}$ respectively. 
\item  $Q/q \to 1$ in probability as $q \to \infty$. 
\end{enumerate}
\end{theorem}

We will also prove a stronger form of Theorem~\ref{thmCritical} part
(2), showing that this holds for all locally tree-like graphs in
$\cG_{\Delta, \del}$ for $\beta$ in an interval around $\beta_c$, see
Proposition~\ref{propLocalConvergence}.  

Theorem~\ref{thmCritical} is derived via the following result, in
which $f_{\dis}$ and $f_{\ord}$ are the functions from
Theorem~\ref{thmPhases} part (1). In particular, the result will
  be used to characterize
the distribution of the random variable $Q$ from
Theorem~\ref{thmCritical}. The functions $\alpha_{k}^{\dis}$
and $\alpha_{k}^{\ord}$ in the theorem statement depend on $q, \Delta, \beta$ and are defined in
Section~\ref{sec:finite-size-scaling}.
\begin{theorem}
  \label{thm:Q}
  Fix $\Delta\geq 5$. Let $Y_1, Y_2, \dots$ be a sequence of
  independent Poisson random variables where $Y_k$ has mean
  $(\Delta-1)^{k}/(2k)$. For the random $\Delta$-regular  graph $G_n$
  and $q=q(\Delta)$ large enough,
  \begin{enumerate}
  \item For $\beta < \beta_c$,
    $\log Z_{G_{n}} - n f _{\dis}$ converges in
    distribution to $W^{\dis}$ given by the almost surely absolutely
    convergent series
    \begin{equation*}
      W^{\dis} \bydef \sum_{ k \ge 3}  \alpha_k^{\dis} Y_k \,.
    \end{equation*}
  \item For $\beta > \beta_c$,
    $\log Z_{G_{n}} - n f _{\ord}$ converges in
    distribution to $W^{\ord}$ given by the almost surely absolutely
    convergent series
    \begin{equation*}
      W^{\ord} \bydef \log q +  \sum_{ k \ge 3}  \alpha_k^{\ord} Y_k \,.
    \end{equation*}
  \item For $\beta =\beta_c$,
  \begin{equation*}
    \frac{Z_{G_n}}{ \exp \left(  n f_{\dis} (\beta_c) \right)}
    \longrightarrow \exp (W^{\dis} )+ \exp (W^{\ord})
    \end{equation*}
   in probability as $n \to \infty$. 
\end{enumerate}
\end{theorem}
The random variable $Q$ in Theorem~\ref{thmCritical} is $\exp(W^{\ord}-W^{\dis})$ with $\beta=\beta_c$. Theorem~\ref{thm:Q} can be viewed as a kind of  finite-size scaling
result, as the random variables $W^{\dis}$ and $W^{\ord}$ capture the
deviations in $ Z_{G_n}$ from the bulk tree-like
behavior. Theorem~\ref{thmCritical} is a comparison of the size of
these corrections for the ordered and disordered contributions to the
partition function. 

The Poisson random variables $Y_k$ in Theorem~\ref{thm:Q}
correspond to the limiting distribution of the number of cycles of
length $k$ in the random regular graph.  Similar Poisson random
variables arise  in analogous results for a class of
random constraint satisfaction problems (including random graph
coloring) in the replica symmetric regime that have been obtained by
combining the small subgraph conditioning method with the
second-moment method or rigorous implementations of the cavity method~\cite{coja2018charting,rassmann2019number,coja2020replica}.  See
also~\cite{montanari2005compute,chertkov2006loop,lucibello2014finite}
for more on finite-size effects in spin models on random graphs and
corrections to the Bethe formula due to short cycles.
 
\begin{rmk}
  For integer $q$, the results of Theorems~\ref{thmPhases}
  and~\ref{thmCritical} can be immediately transferred to the Potts
  model via the Edwards-Sokal coupling. In particular this means that
  the function $\beta_c(q,\Delta)$ in Theorem~\ref{thmPhases} must
  agree with~\eqref{eqBetaCexplicit} for integer $q$.
\end{rmk}

\begin{rmk}
  The lower bound $\Delta \ge 5$ in Theorems~\ref{thmPhases}
  and~\ref{thmCritical} and Theorem~\ref{thmAlgorithm} below
  facilitates some arguments involving small-set expansion.  We
  believe our results could be extended to $\Delta=3,4$ with a more
  delicate analysis.
\end{rmk}

While we are able to obtain a much more complete description of the
phase diagram of the random cluster and Potts models on random
$\Delta$-regular graphs than in previous works, we emphasize that our
techniques are inherently perturbative, i.e., rely on taking $q$
large. In particular, our arguments require $q \ge \Delta^{C \Delta}$
for some fixed, but large, $C$.  It would be very interesting to
extend these results to all $q > 2$.

\subsection{Approximate counting and sampling}
\label{secAlgorithmic}

One motivation for 
this paper is to investigate the
relationship between phase transitions and computational complexity.
The family of graphs most central to 
the interplay
between phase transitions and algorithms are arguably random graphs.
Random graphs provide candidate hard instances for several important
NP-hard problems like max independent set, $q$-coloring, and MAX-CUT.
Explanations for the computational hardness of these instances are
given by structural properties of the relevant statistical physics
models (hard-core model, anti-ferromagnetic Ising and Potts models),
including \emph{replica symmetry breaking} of the solution space as
predicted by the cavity method from statistical
physics~\cite{krzakala2007gibbs,achlioptas2008algorithmic,molloy2018freezing}.
On the other hand, another class of models, including the hard-core
model on random bipartite graphs and the ferromagnetic Ising and Potts
models, do not exhibit replica symmetry breaking: they are replica
symmetric over the entire range of parameters~\cite{sly2014counting,dembo2014replica}.  Nonetheless these models still
play an important role in computational complexity: they are used as
gadgets in hardness reductions for approximate counting and
sampling~\cite{sly2010computational,galanis2011improved,sly2014counting,galanis2016inapproximability,galanis2016ferromagnetic}.

In this paper we investigate approximate counting and sampling
problems in  replica symmetric models on random graph through the lens of the $q$-color ferromagnetic
  Potts and random cluster models. These models exhibit a
first-order `disorder/order' phase
transition (proved for the Potts model in~\cite{galanis2016ferromagnetic}, and for the random cluster model in this paper), and this phase transition
has been used in the construction of gadgets to show \#BIS-hardness of
sampling from the Potts model on bounded degree
graphs~\cite{galanis2016ferromagnetic}. But are these instances
computationally hard themselves?  For the case of large $q$, we
establish that the answer is
`no': there are efficient approximate sampling and counting algorithms
for the Potts and random cluster models at \emph{all}
temperatures. Very few all-temperature
  algorithms for statistical physics models on random regular
  graphs are known (examples include the special cases of the Ising and monomer-dimer models for which the problems are tractable on all graphs~\cite{jerrum1989approximating,jerrum1993polynomial}).  To the best of our knowledge this is the first such result for a problem where
  approximate counting is  NP- or \#BIS-hard in the worst case (in this
  case, \#BIS-hard).  Our counting algorithms are deterministic, and  
rely crucially on exploiting the first-order phase transition established in Theorem~\ref{thmPhases}.

There are two main computational problems associated to statistical
physics models like the random cluster and Potts models: the
\emph{counting problem} of computing the partition function $Z$, and
the \emph{sampling problem} of outputting a random configuration
distributed as $\mu$.  In general these problems are \#P-hard even for
restricted classes of graphs and parameter settings.  As a result,
current research is focused on finding efficient \textit{approximate}
counting and sampling algorithms.  We say $\hat Z$ is an
\emph{$\eps$-relative approximation to $Z$} if
$ e^{-\eps} \hat Z \le Z \le e^{\eps} \hat Z$.  A \emph{fully
  polynomial-time approximation scheme} (FPTAS) is an algorithm that,
given a graph $G$ and any $\eps>0$, outputs an $\eps$-relative
approximation to $Z_G$ and runs in time polynomial in $1/\eps$ and
$|V(G)|$.  A \emph{polynomial-time sampling algorithm} is a randomized
algorithm that, given a graph $G$ and any $\eps >0$, outputs a configuration $A$ with
distribution $\hat \mu$ in time polynomial in $|V(G)|$ and $1/\eps$ such that $\| \mu_G - \hat \mu \|_{TV} \le \eps$.

There is an extensive literature on approximate counting and sampling
from the Potts and random cluster models which we discuss below in
Section~\ref{sec:RelatedAlg}. First, however, we state our algorithmic
results.  Recall the class of graphs $\cG_{\Delta,\del}$ from Section~\ref{sec:expansion-profiles}.

\begin{theorem}
  \label{thmAlgorithm}
  For every $\Delta \ge 5$, $\del >0$, and $q$ large enough as a
  function of $\Delta, \del$, there is an FPTAS and a polynomial-time
  sampling algorithm for the $q$-color Potts and random cluster models
  at all inverse temperatures $\beta \ge 0$ over the class
  $\cG_{\Delta, \del}$.
\end{theorem}

\begin{cor}
  \label{thmAlgRandom}
  For $\Delta \ge 5$ and $q= q(\Delta)$ large enough, with high
  probability over the random $\Delta$-regular graph, there is a FPTAS
  and polynomial-time sampling algorithm for the $q$-color Potts and
  random cluster models at all temperatures.
\end{cor}
In particular, there is an algorithm that in polynomial-time decides
to accept or reject a random $n$-vertex,  $\Delta$-regular graph $G$. The
acceptance condition is simply that $G \in \cG_{\Delta, \del}$, see
Lemma~\ref{prop:RRG-exp} below. For accepted graphs, the algorithm of
Theorem~\ref{thmAlgorithm} can be used.

Our methods also allow us to obtain negative algorithmic results,
i.e., to establish exponentially slow mixing of some well-known Markov
chains. Precise definitions of the Markov chains appearing in the next
theorem, and of mixing times, can be found in Section~\ref{secSlow}.

The Swendsen--Wang dynamics~\cite{swendsen1987nonuniversal} are
non-local dynamics for the Potts model devised to circumvent the
problem of phase coexistence by allowing re-coloring of many vertices
in a single step of the chain.  On the lattice $(\Z/n \Z)^d$ (with $q$
sufficiently large) the Swendsen--Wang dynamics are expected to be
fast except at criticality, where the mixing time is
$\exp(\Omega(n^{d-1}))$~\cite{borgs2012tight,gheissari2018mixing}. On
the other hand, for the mean-field model (i.e., on the complete graph)
the mixing time is exponentially slow in an entire interval around
$\beta_c$~
\cite{gore1999swendsen,blanca2015dynamics,galanis2019swendsen,gheissari2020exponentially}.
It has been conjectured that the Swendsen--Wang dynamics for random
regular graphs exhibit mean-field behavior, mixing exponentially
exponentially slowly for $q>2$ and $\beta$ in the entire interval
$( \beta_u, \beta_{u}^{\star})$. See~\cite{blanca2021random} for a
discussion. The next theorem takes a step towards confirming that the
mean-field picture is correct for random regular graphs by proving
slow mixing in an interval. Previously slow mixing was only known at
criticality, consistent with both lattice and mean-field-type
behavior~\cite{galanis2016ferromagnetic,blanca2021random}.

  \begin{theorem}
  \label{thmSlowMix}
  For all $\Delta\geq 5$ and $q =q(\Delta)$ large enough, there is an
  interval $(\beta_m, \beta_M)$ containing $\beta_c$ so that for any
  sequence of locally tree-like graphs $G_n \in \cc G_{\Delta,\delta}$, the mixing times of the random cluster Glauber dynamics and the Swendsen--Wang dynamics\footnote{For non-integer $q$ we consider the Chayes-Machta dynamics~\cite{chayes1998graphical}, a generalization of Swendsen--Wang to non-integer $q$.  The results also apply to the Potts Glauber dynamics.} are $e^{\Omega(n)}$.  
  \end{theorem}
  As remarked above, slow mixing was previously only known at the
  critical point $\beta = \beta_c$ on random $\Delta$-regular graphs
  for $q\geq 2\Delta/\log\Delta$~\cite{galanis2016ferromagnetic}.  Our
  ability to prove Theorem~\ref{thmSlowMix} is due to the fact that
  our methods give us detailed information about the relative
  probabilities of ordered and disordered configurations in an
  interval around $\beta_{c}$.

\subsubsection{Related algorithmic work}
\label{sec:RelatedAlg}

For the special case of the Ising model ($q=2$), Jerrum and
Sinclair~\cite{jerrum1993polynomial} gave a fully polynomial-time
randomized approximation scheme (FPRAS) for all graphs and all inverse
temperatures $\beta \ge 0$.  See also~\cite{guo2018}, which shows that
the $q=2$ random cluster dynamics are rapidly mixing on all
graphs. When $q=1$ the random cluster model is Bernoulli bond
percolation, and algorithmic tasks are trivial. The existence of
efficient algorithms at all temperatures and on all graphs appears to
be a rather special property, though. For $q\notin\{1,2\}$ most
positive algorithmic results to date have been restricted to the
high-temperature regime ($\beta$ small) or, when $q$ is integral, the
low-temperature regime ($\beta$ large).

The most relevant results for the present paper are that, for the case
of random
$\Delta$-regular graphs, Blanca, Galanis, Goldberg,
{\v{S}}tefankovi{\v{c}}, Vigoda, and Yang~\cite{blanca2020sampling}
give an efficient sampling algorithm for the Potts and random cluster
models when the temperature is above the uniqueness threshold of the
infinite $\Delta$-regular tree, i.e., $\beta < \beta_u(\mathbb
T_\Delta)$; Blanca and Gheissari then showed $O(n \log
n)$ mixing of the random cluster dynamics for
$\beta<\beta_u$~\cite{blanca2021random}.  In a more general setting,
Bordewich, Greenhill, and Patel showed the Glauber dynamics for the
Potts model on graphs of maximum degree at most
$\Delta$ mix rapidly for $\beta \le (1+o_q(1))\log
q/(\Delta-1)$ and showed slow mixing of the Glauber dynamics on random
regular graphs for $\beta \ge (1+o_q(1))\log q/(\Delta-1
-1/(\Delta-1))$~\cite{bordewich2016mixing}. See also~\cite{CDKPR20}
for deterministic algorithms in a slightly smaller range of $\beta$.

On the hardness side, Goldberg and
Jerrum~\cite{goldberg2012approximating} showed that approximating the
Potts model partition function $Z^{\Potts}$ 
for $q\ge 3$ on general graphs is \#BIS-hard; that is, it
is as hard as approximating the number of independent sets in a
bipartite graph~\cite{dyer2004relative}.  Galanis,
{\v{S}}tefankovi{\v{c}}, Vigoda, and
Yang~\cite{galanis2016ferromagnetic} refined these results by showing
that for $\beta >\beta_{c}(q,\Delta)$ (recall~\eqref{eqBetaCexplicit}), 
it is \#BIS-hard to approximate $Z^{\Potts}$ on graphs of maximum
degree $\Delta$. The description of the phase diagram of the Potts
model on random regular graphs found
in~\cite{galanis2016ferromagnetic} and discussed above is a crucial
ingredient for this result.

For $q>2$ the algorithms mentioned above apply only in the
high-temperature regime.  At very low temperatures 
efficient sampling and counting algorithms have
recently been developed for structured classes of
graphs~\cite{HelmuthAlgorithmic2,barvinok2019weighted,huijben2021sampling} and for
expander graphs~\cite{JKP2,galanis2020fast,carlson2020efficient,galanis2021fast}. By making use of the
ideas in~\cite{HelmuthAlgorithmic2} in combination with the
application of Pirogov-Sinai theory to the random cluster model
in~\cite{borgs2012tight}, Borgs, Chayes, Helmuth, Perkins, and
Tetali~\cite{PottsAll2019stoc} gave efficient counting and sampling
algorithms for the random cluster model on $d$-dimensional tori
$(\Z/n\Z)^{d}$ at \emph{all} temperatures when $q$ is large enough as
a function of $d$.  As the tori $(\Z/n\Z)^{d}$ approximate
  $\Z^{d}$ as $n\to\infty$, the results of~\cite{PottsAll2019stoc} can be
interpreted as saying that the phase transition for the $q$-state
random cluster model on $\Z^{d}$ is not an algorithmic barrier, at
least when $q$ is large. Our Theorem~\ref{thmAlgorithm} has a
similar interpretation: informally speaking, it says that the phase transition for
the $q$-state random cluster model on random $\Delta$-regular graphs
is not an algorithmic barrier when $q$ is large.  

On the other hand, it is shown in~\cite{galanis2016ferromagnetic} that
this phase transition does have a link to computational complexity,
since it can be used to construct gadgets which show the
\#BIS-hardness of approximating the Potts partition function for
$\beta > \beta_c$.  Similarly, while the phase transition in the
hard-core model on random $\Delta$-regular bipartite graphs has a
direct link to the NP-hardness of approximating the independence
polynomial on bounded degree
graphs~\cite{sly2014counting,sly2010computational,galanis2016inapproximability,weitz2006counting},
it is still plausible that there are efficient sampling and counting
algorithms for the hard-core model on random bipartite graphs at all
activities.  For activities large enough, efficient algorithms are given in~\cite{JKP2,liao2019counting,chen2021sampling}, with the last paper obtaining the best known bound of $\lam = \Omega(\log \Delta/ \Delta)$.  The authors of~\cite{chen2021sampling} observe, however, that $ \Omega(\log \Delta/ \Delta)$ appears to be a barrier for the type of polymer model argument used in these papers, and so finding efficient algorithms for the hard-core model at and slightly above criticality will likely require new ideas and techniques.

\subsection{Open problems}
\label{sec:previous}

For integer $q$, the explicit formula~\eqref{eqBetaCexplicit} for
$\beta_c(q,\Delta)$ was found in~\cite{galanis2016ferromagnetic}.  We conjecture that this formula holds for non-integral $q$ as well:
\begin{conj}
  \label{conjBc}
  For all $\Delta \ge 3$, $q>2$, the critical point of the random
  cluster model on the random $\Delta$-regular graph satisfies
  \begin{equation*}
    \beta_c(q,\Delta) = \log  \frac{  q-2 }{ (q-1)^{1-2/\Delta} -1   }  \,.
  \end{equation*}
\end{conj}
It would be interesting to see if the methods of~\cite{galanis2016ferromagnetic,dembo2014replica} could be generalized to the random cluster model to prove this conjecture.

As discussed earlier, our methods rely in an essential way on $q$
being very large.  Phenomenologically, one
expects the same behavior for all $q>2$.  It would be very interesting to extend our results to this
setting. One possible approach, at least for algorithmic results, is
to show that a Markov chain started from either $A= \emptyset$ or
$A =E$ can be used to obtain approximate samples.

\subsection{Organization of the paper}

In Section~\ref{sec:polym-model-repr} we define ordered and disordered
polymer models, and prove estimates showing that their cluster
expansions converge in overlapping regions of parameters that cover
all inverse temperatures $\beta$ when $q$ is large as a function of
$\Delta$.

The definitions and estimates of Section~\ref{sec:polym-model-repr}
allows us to study the random cluster model via the polymer models.
In Section~\ref{sec:phase-transition} we exploit this polymer model
framework to prove Theorem~\ref{thmPhases}. In
Section~\ref{sec:RRGexp} we show that the random $\Delta$-regular
graph belongs to $\cG_{\Delta,\del}$ with high probability and that
membership in $\cG_{\Delta,\del}$ can be checked efficiently. In
Section~\ref{sec:finite-size-scaling} we prove
Theorems~\ref{thmCritical} and~\ref{thm:Q}.  In Section~\ref{secSlow}
we prove Theorem~\ref{thmSlowMix}.

The estimates of Section~\ref{sec:polym-model-repr} reduce the
algorithmic problems to algorithmic problems concerning 
the two polymer models.  In Section~\ref{sec:algor-cons} we provide
efficient approximate counting and sampling algorithms for the polymer
models, proving Theorem~\ref{thmAlgorithm}.

\section{Polymer model representations and estimates}
\label{sec:polym-model-repr}

Recall the definition of $\cG_{\Delta,\del}$ from
Section~\ref{sec:expansion-profiles}. In this section we assume the following.
\begin{assm}
  \label{assm}
  We assume $q \ge1$, $\Delta \ge 5$,  $G\in \cG_{\Delta,\del}$, and set parameters
  \begin{equation}
    \label{eq:assm}
    n \bydef \abs{V(G)}, \quad e^{\beta_{0}}-1\bydef q^{1.9/\Delta}, \quad
    e^{\beta_{1}}-1\bydef q^{2.1/\Delta}, \quad \eta \bydef \min \{1/100,\delta/5\}.
  \end{equation}
\end{assm}

\subsection{\emph{A priori} estimates}
\label{sec:an-empha-priori}

The intuition for understanding the large-$q$ random cluster model on
an expander graph is that at \textit{all} temperatures a typical configuration
consists of either very few of the edges or nearly all of the edges.
The following  lemma will allow us to make this
precise. Recall $c(A)$ is the number of connected components induced
by an edge set $A$.

\begin{lemma}
\label{lemNoMiddle}
Suppose $G \in \cG_{\Delta, \del}$ and $n \ge 360/(\eta\del)$. For $A \subset E$ such that
$\eta |E| \le |A| \le (1-\eta)|E|$,
\begin{equation*}
\frac{c(A)}{n} + \frac{|A|}{|E|} \le 1-\eta/40 \,.
\end{equation*}
\end{lemma} 
\begin{proof}
  Given $A$, let $n_1$ denote the number of vertices in connected
  components of size $1$ and $n_2$ the number in components of size at
  least $2$ and at most $\del n$. We will prove the lemma by a case
  analysis based on $n_{1}+n_{2}$. Before doing this, we
  collect some useful facts.  

  First, observe that
  \begin{equation}
    \label{eq:cA}
    c(A) \le n_1 + n_2/2 + 1/\del  \,.
  \end{equation}
  Second, the number of unoccupied edges satisfies
  \begin{equation}
    \label{eqEAlb}
    |E| - |A| \ge n_1 \frac{\Delta}{2} + n_2 \frac{5}{9} \frac{\Delta}{2}+ \min \{ n_1 +n_2, n-n_1-n_2 \} \cdot  \frac{\Delta} {20} \,
  \end{equation}
  by double counting the unoccupied edges and the definition of
  $G_{\Delta,\delta}$.  Together~\eqref{eq:cA}, \eqref{eqEAlb},
  $|E| = n\Delta/2$, and a little algebra gives
  \begin{equation}
    \label{eqCaAE2}
    \frac{c(A)}{n} + \frac{|A|}{|E|} \le 1 - \frac{1}{18} \frac{n_2}{n}-   \frac{  \min \{ n_1+n_2, n-n_1 - n_2 \} }{ 10 n  }    + \frac{1}{\del n} \,.
  \end{equation}
  
  We now perform the case analysis. First, if
  $\frac{n_1 + n_2}{n} \le \eta/2$, then since $|A|/|E| \le 1- \eta$
  and $c(A)\leq n_{1}+n_{2}+(\delta n)^{-1}$,
  we directly see that 
  \begin{equation*}
    \frac{c(A)}{n} + \frac {|A|}{|E|}  \le \frac{\eta}{2} + 1- \eta + \frac{1}{\del n} \le 1- \frac{\eta}{2} + \frac{1}{\del n} \,.
  \end{equation*}

  Next suppose $\frac{n_1 + n_2}{n} > 1- \eta /2$.  
  If $n_{2}/n<\eta/2$, then $n_{1}/n >
  1-\eta$. This implies $(|E| - |A|)/|E| > 1- \eta$, or $|A| < \eta |E|$ which contradicts
  the assumption on $A$. 
  We may therefore assume that $n_{2}/n \geq \eta/2$
  and so \eqref{eqCaAE2} gives
  \[
  \frac{c(A)}{n} + \frac {|A|}{|E|} \le 1-  \frac{\eta}{36} + \frac{1}{\del n}\, .
  \]

  Finally if $ \eta /2 \le \frac{n_1 + n_2}{n} \le 1- \eta/2$
  then~\eqref{eqCaAE2} gives
  \begin{equation*}
    \frac{c(A)}{n} + \frac {|A|}{|E|}  \le 1-   \frac{\eta}{20} + \frac{1}{\del n} \,.
  \end{equation*}
  In each case we've shown that 
  $\frac{c(A)}{n} + \frac {|A|}{|E|} \le 1- \frac{\eta}{36} + \frac{1}{\del n} $,
  and the lemma follows.
\end{proof}

Recall that $\Omega = \{0,1\}^{E}$, and that
$Z = \sum_{A\in\Omega}q^{c(A)} (e^\beta -1)^{|A|}$. When we write
$A\in\Omega$ we call the edges in $A$ \emph{occupied} and the edges
not in $A$ \emph{unoccupied}.  In light of Lemma~\ref{lemNoMiddle}, we
decompose the state space into three pieces.  Recall that
\begin{align*}
\Omega_{\dis} &\bydef \{ A \in \Omega : |A| \le \eta |E| \} \,  \\
\Omega_{\ord} &\bydef \{ A \in \Omega : |A| \ge (1-\eta) |E| \} \\
\Omega_{\err} &\bydef \Omega \setminus (\Omega_{\dis} \cup \Omega_{\ord}) \,.
\end{align*}
Let $Z^{\ord}, Z^{\dis}$, $Z^{\err}$ be the corresponding random
cluster model partition functions. Thus
\begin{align}
  \label{eq:Zdecomp}
  Z = Z^{\ord} + Z^{\dis} + Z^{\err}\, .
\end{align}
The next lemma shows that $Z^{\err}$ represents an exponentially small
fraction of the partition function. It also establishes that unless
$\beta$ lies in the interval $(\beta_0, \beta_1)$, then up to an
exponentially small correction, $Z$ is given by $Z^{\ord}$ or $Z^{\dis}$.

\begin{lemma}
  \label{lem:ErrQual}
  If $q$ and $n$ are sufficiently large as a function of $\Delta$ and $\del$ then the following hold.
  \begin{enumerate}
\item For $\beta \ge0$, $Z^{\err}/Z \le e^{-n}$. 
\item For $\beta \ge \beta_1$,  $Z^{\dis}/Z \le e^{-n}$.
\item For $\beta \le \beta_0$,  $Z^{\ord}/Z \le e^{-n}$.
\end{enumerate}
\end{lemma}
\begin{proof}
  Let $z\bydef \max\left\{q, (e^{\beta}-1)^{\Delta/2}\right\}$. Then
  $Z \geq z^n$ so that for $A\in \Omega$,
  \begin{equation*}
    \frac{q^{c(A)} (e^{\beta} -1) ^{| A|}}{Z} \leq z^{2|A|/\Delta + c(A)-n}\, .
  \end{equation*}
  By Lemma~\ref{lemNoMiddle}, we have for every $A \in \Omega_{\err}$,
  \begin{equation*}
    \frac{2|A|}{\Delta} + c(A)-n\leq  - \frac{\eta n }{40} \,.
  \end{equation*}
  Since $\abs{\Omega_{\err}}\leq |\Omega|=2^{n\Delta/2}$, it follows that 
  \begin{equation*}
    Z^{\err}/Z\leq 2^{n\Delta/2}z^{-\eta n/40} \le 2^{n\Delta/2}q^{-\eta  n/40}  \, ,
  \end{equation*}
  which proves part (1) for $q=q(\Delta,\delta)$ large enough. 
    
  Next suppose $\beta \ge \beta_1$.  To prove part (2), recall that
  $\eta\leq 1/100$, so using $|\Omega_{\dis}|\leq 2^{n\Delta/2}$,
  \begin{equation*}
    \frac{ Z^{\dis}} { Z} \le (e^{\beta} -1)^{-\Delta n/2} Z^{\dis} \le  (e^{\beta} -1)^{(\eta -1) \Delta n/2} 2^{\Delta n/2} q^n    \le q^{-2.1 \cdot .99 n/2  +n } 2^{\Delta n/2}   \le e^{-n}
  \end{equation*}
  for $q = q(\Delta)$ large enough.

  Lastly, suppose $\beta \le\beta_0$.  Then, using that
  $A\in\Omega_{\ord}$ implies $c(A)$ is at most $\eta n$,
  \begin{equation*}
    \frac{Z^{\ord}}{Z} \le  q^{-n} Z^{\ord} \le q^{(\eta-1)n}
    2^{\Delta n/2} (e^{\beta_{0}}-1)^{\Delta n/2}  \le q^{-.99 n + .95 n}
    2^{\Delta n/2}\le e^{-n} 
  \end{equation*}
  for $q= q(\Delta)$ large enough, which proves part (3).
\end{proof}

Lemma~\ref{lem:ErrQual} implies the contribution of $Z^{\err}$ to $Z$
is negligible at all temperatures, and so it suffices to control $Z^{\dis}$, $Z^{\ord}$
or both, depending on the value of $\beta$.  We will do this by
defining two polymer models and proving they have convergent cluster
expansions for $\beta \in [0, \beta_1 ]$ and
$\beta \in [\beta_0, \infty)$ respectively. Crucially, since
$\beta_{0}<\beta_{1}$, these two intervals overlap.

\subsection{Polymer models}
\label{sec:polymer-models}

Let $\cP$ be a collection of (possibly edge-labelled) finite connected
subgraphs of some given finite or infinite graph.  We refer to the elements of $\cP$ as
\emph{polymers}.  We say that two polymers $\gamma_1, \gamma_2\in \cP$
are \emph{compatible}, denoted $\gamma_1 \sim \gamma_2$, if they are
vertex disjoint, and we write $\gamma_{1}\nsim\gamma_{2}$ to denote
incompatibility.   Let $w\colon \cP\to \mathbb C$; $w$ is called a
\emph{weight function}.  The triple $(\cP, \sim, w)$ is a special case
of a \emph{polymer model} as defined by Koteck\'{y} and
Preiss~\cite{kotecky1986cluster}, generalizing a technique used to
study statistical mechanics models on lattices, see, e.g.,~\cite{gruber1971general}.  

Let $\cP' \subseteq \cP$ be a finite subset of polymers, and let
$ \Omega(\cP')$ denote the family of all sets of pairwise compatible
polymers from $\cP'$. Then the expression
\begin{equation*}
\Xi(\cP')\bydef \sum_{\Gamma \in \Omega(\cP')} \prod_{\gamma\in \Gamma}w(\gamma)
\end{equation*}
is the \emph{partition function} of the polymer model $(\cP', \sim,w)$.  The cluster
expansion is a formal power series for $\log \Xi(\cP')$.  In order to
describe the cluster expansion we require some notation.

Suppose that $\Gamma=(\gamma_1, \ldots, \gamma_t)$ is an ordered tuple
of polymers.  We define the \emph{incompatibility graph} $H_{\Gamma}$
to be the graph on the vertex set ${1,\ldots, t}$ where
$\{i,j\}\in E(H_{\Gamma})$ if and only if $i \neq j$ and $\gamma_i$ is
incompatible with $\gamma_j$.  A \emph{cluster} is an ordered tuple
$\Gamma$ of polymers whose incompatibility graph $H_\Gamma$ is
connected.  Given a graph $H$, the \emph{Ursell function} $\phi(H)$ of
$H$ is
\begin{equation*}
\phi(H) \bydef \frac{1}{|V(H)|!} \sum_{\substack{A \subseteq E(H)\\ \text{spanning, connected}}}  (-1)^{|A|} \, .
\end{equation*}

Let $\cC$ be the set of all clusters of polymers from $\cP'$. The
\emph{cluster expansion} is the formal power series in the weights
$w(\gamma)$
\begin{equation}
  \label{eq:ce}
  \log \Xi(\cP') = \sum_{\Gamma \in \cC}  w(\Gamma) \,,
\end{equation}
where 
\begin{equation*}
  w(\Gamma)\bydef  \phi(H_{\Gamma})
  \prod_{\gamma\in\Gamma}w(\gamma)\, .
\end{equation*}
The convergence of the infinite series on the right-hand side
of~\eqref{eq:ce} is not automatic. The following theorem
gives a convenient condition for convergence, and a useful
consequence. 

Let $E(\gam)$ denote the set of edges in the polymer $\gam$. For a
cluster $\Gamma$ let $\| \Gamma\| \bydef \sum_{\gam \in \Gamma} |E(\gam)|$
and write $\Gamma \nsim \gam$ if there exists $\gam' \in \Gamma$ so
that $\gam \nsim \gam'$.

\begin{theorem}[\cite{kotecky1986cluster}]
\label{thmKP}
Suppose that there exists $r \ge 0$ such that for all polymers $\gam \in \cP$, 
\begin{equation}
\label{eqKPcond}
\sum_{\gam' \nsim \gam}  |w(\gam')| e^{(1+r)|E(\gam')|}  \le |E(\gam)| \,,
\end{equation}
then the cluster expansion for $\log \Xi(\cP')$ converges absolutely for every finite subset $\cP' \subseteq \cP$.  Moreover, for all polymers $\gam$,
\begin{equation}
  \label{eqKPtail}
  \sum_{\Gamma \in \cC, \;  \Gamma \nsim \gam} \left |  w(\Gamma) \right| e^{r\|\Gamma\|} \le |E(\gam)| \,.
\end{equation}
\end{theorem}

Our applications of polymers models will involve weights $w(\gamma)$
that are analytic functions of a parameter $\beta$. By verifying
that~\eqref{eqKPcond} holds uniformly for all $\beta$ in a domain in
the complex plane, we will obtain analyticity of $\log \Xi$ in the same
domain, as Theorem~\ref{thmKP} then implies that the right-hand side
of~\eqref{eq:ce} converges uniformly in $\beta$ in the domain.

Note that when the weights $w(\gamma)$ of a polymer model are all
non-negative reals, we can define an associated Gibbs measure $\nu$ on
$\Omega(\cP')$ by
\begin{equation}
\label{eqNuGamma}
 \nu(\Gamma) \bydef \frac{ \prod_{\gamma\in \Gamma}w(\gamma)} {\Xi(\cP')} \,.
\end{equation}

\subsection{Disordered polymer model}
\label{sec:dis}

In this section we describe a polymer model that captures deviations
from the disordered ground state $A_{\dis} = \emptyset$. 

Define \emph{disordered polymers} to be connected subgraphs $(V',E')$ of $G$ with
$|E'| \le \eta n$. Let $\cP_{\dis}=\cP_{\dis}(G)$ be the set of disordered polymers in $G$. 
Two polymers are compatible if they are vertex disjoint. For a polymer
$\gamma$, let $|\gamma|$ denote the number of vertices of $\gamma$ and
$| E( \gamma )|$ the number of edges.  The weight of the polymer is
defined to be
\begin{equation*}
  w^{\dis}_\gamma \bydef q^{1- |\gamma|} (e^{\beta} -1)^{|E(\gamma)|}.
\end{equation*}
The \emph{disordered polymer partition function } is
\begin{equation*}
\Xi^{\dis} \bydef \sum_{\Gamma}
\prod_{\gamma\in\Gamma}w^{\dis}_{\gamma},
\end{equation*}
where the sum is over all compatible collections of disordered
polymers.

\begin{prop}
  \label{prop:DisPoly}
  If $q$ and $n$ are sufficiently large as a function of $\Delta$ and $\del$,
   then for all $\beta\in \mathbb C$ such that 
   $\left| e^{\beta}-1\right| \le e^{\beta_{1}}-1$,
   the disordered polymer model
  satisfies~\eqref{eqKPcond} with $r=\log q/(4\Delta)$.
\end{prop}
\begin{proof}
  We will show that for $\beta \le \beta_1$ and for every
$v \in V(G)$,
\begin{equation}
\label{eqKPgoalDis}
\sum_{\gamma \ni v} e^{(1+r)|E(\gamma)|} \left| w^{\dis}_\gamma \right| \le  \frac{1}{ 2} \,.
\end{equation}
This is sufficient to verify~\eqref{eqKPcond} for the disordered
polymer model: given a polymer $\gamma'$, sum
\eqref{eqKPgoalDis} over all vertices of $\gamma'$. Since
$|\gam'|/2\leq |E(\gam')|$, we obtain~\eqref{eqKPcond}. We will
prove~\eqref{eqKPgoalDis} in three steps.

First we consider polymers with $|E(\gamma)|=1$ and $|E(\gamma)|=2$.
The contribution to the left-hand side of~\eqref{eqKPgoalDis} from such polymers
is exactly
$e^{1+r}\Delta |e^{\beta} -1| q^{-1} + \frac{3}{2} e^{2+2r}\Delta (\Delta-1)
|e^{\beta} -1|^2 q^{-2} $.  Since $\Delta\geq5$ and $\abs{e^{\beta} -1}\leq q^{\frac{2.1}{\Delta}}$,
this is at most $1/6$ for $q=q(\Delta)$ large enough.

Next we consider polymers with $2 < |E(\gamma)| < \Delta/2$.  Note
that $|\gamma|\geq \sqrt{2|E(\gamma)|}$ for any polymer.   By~\cite[Lemma~2.1
(c)]{BorgsChayesKahnLovasz} we can bound the number of polymers with $k$ edges
containing a fixed vertex $v$ by $(e\Delta)^{k}$.   We bound the
contribution of these polymers to the left-hand side of~\eqref{eqKPgoalDis} by
\begin{equation*}
\sum_{ 3 \le k < \Delta/2}   (e^{2+r} \Delta)^{k} q^{1 -\sqrt{2k}    }  \abs{e^{\beta} -1}^k \le \sum_{ 3 \le k < \Delta/2}   (e^{2+r} \Delta)^{k} q^{1 -\sqrt{2k}  +\frac{2.1}{\Delta}k    }  
\le   \sum_{ 3 \le k < \Delta/2}   |e^{2+r} \Delta|^{k} q^{2.05 -\sqrt{2k}    } \,.
\end{equation*}
Since $\sqrt{6}-2.05>1/8$, this is at most $1/6$ for
$q=q(\Delta)$ large enough.   

For larger polymers we need two facts. First, that
$\Delta\abs{\gamma} \geq 2\abs{E(\gamma)} + \abs{E(V(\gamma),
  V(\gamma)^c)}$.  Second, since $\gamma$ defines a connected subgraph
we have $|\gamma| \le |E(\gamma)| +1 \le 2\eta n \leq \del n$, and so
the vertices of $\gamma$ satisfy the small set expansion condition
guaranteed by $G\in \cc G_{\Delta,\delta}$. Together these facts imply
$|\gamma| \ge\frac{ 9|E(\gamma)|}{ 2\Delta}$.  Using~\cite[Lemma~2.1
(c)]{BorgsChayesKahnLovasz} we bound the contribution of these
polymers to the left-hand side of~\eqref{eqKPgoalDis} by
\begin{equation*}
     \sum_{k \ge \Delta/2} (e^{2+r} \Delta)^{k}
    q^{1 - \frac{9k}{2\Delta}     }  \abs{e^{\beta} -1}^k 
                                                        \le  \sum_{k \ge \Delta/2} (e^{2+r} \Delta)^{k}
    q^{1 - \frac{9k}{2\Delta}  + \frac{2.1 k}{\Delta}  }  
       =  \sum_{k \ge \Delta/2} (e^{2+r} \Delta)^{k}
    q^{1 - \frac{12 k}{5\Delta}    } \, ,
\end{equation*}
and again this is at most $1/6$ for $q=q(\Delta)$ large enough.
\end{proof}

We remark that it is useful to consider complex $\beta$ here in order 
to derive analyticity properties of limiting free energies later (see for example Lemma~\ref{prop:limf} below).

\subsubsection{Disordered polymer measure on edges}
\label{sec:disord-polym-meas}
In addition to the Gibbs measure for the disordered polymer model
given by~\eqref{eqNuGamma}, the disordered polymer model also defines
a probability measure $\overline \nu_{\dis}$ on $\Omega=\{0,1\}^E$ via
projection.  To obtain a sample $A$ from $\overline \nu_{\dis}$, first
sample a configuration of compatible disordered polymers $\Gamma$ from
$\nu_{\dis}$, that is, with probability
$\prod_{\gamma \in \Gamma} w^{\dis}_\gamma / \Xi^{\dis} $.  Second,
let
\begin{equation*}
A = \bigcup_{\gamma \in \Gamma} E(\gamma) \, .
\end{equation*}
We will show in Section~\ref{sec:accur-order-expans} that when the
disordered cluster expansion converges, i.e., for $\beta \le \beta_1$,
the distribution $\overline \nu_{\dis}$ is very close to the
distribution of the random cluster model measure $\mu$ conditioned on
$\Omega_{\dis}$.

\subsection{Ordered polymer expansion}
\label{sec:ord}

Next we define a polymer model that describes deviations
from the ordered ground state $A_{\ord}=E$. We need a more complicated
construction compared to the disordered polymer model to
handle the non-local cluster weight. The basic idea of the ordered
polymer model is that, given an edge configuration $A$, polymers represent
the connected components of the `boundary' $\cB(A)$ of $A$. We begin
by making this precise.

\subsubsection{Boundary of occupied edges}
\label{sec:bound-occup-edges}

The precise notion of boundary is given by the following
construction. Given $A\subseteq E$, let $\cB_0(A)$ be the set of unoccupied edges
$E \setminus A$.  To form $\cB_{i+1}(A)$ from $\cB_i(A)$ we add any
edge $e$ incident to a vertex $v$ with at least $5\Delta/9$ incident
edges in $\cB_i(A)$. This procedure stabilizes and results in a set
$\cB_{\infty}(A)$ of edges, of which $\cB_{0}(A)$ are unoccupied and
$\cB_{\infty}(A)\setminus\cB_{0}(A)$ are occupied.

\begin{lemma}
  \label{lem:ord-gen}
  For any $A\subseteq E$,
  the algorithm to generate $\cB_{\infty}(A)$ runs in time quadratic in
  $\abs{\cB_{0}(A)}$. Moreover,
  $\abs{\cB_{\infty}(A)}\leq 10\abs{\cB_{0}(A)}$.
\end{lemma}
\begin{proof}
  First, observe that the same set $\cB_{\infty}(A)$ results no matter
  the order in which edges are added. The first claim therefore will
  follow from the second as there are at most $10\abs{\cB_{0}(A)}$
  edges added, and finding the next edge to add (if one exists) takes
  time at most $10\Delta\abs{\cB_{0}(A)}$.

  To prove the second claim, observe that each incident edge at a
  vertex can be `charged' at most $4/5=(4\Delta/9)/(5\Delta/9)$ for
  the new edges added to the boundary. Since each edge in $\cB_0(A)$
  is incident to $2$ vertices it can be charged at most $8/5$ in
  total, and each subsequent edge can be charged at most $4/5$.  Thus
  \begin{equation*}
    |\cB_{\infty}(A)| \le 2 | \cB_0(A)| \sum_{j=0}^\infty (4/5)^j = 10 |\cB_0(A)| \,.\qedhere
  \end{equation*}
\end{proof}

\subsubsection{Ordered polymers}
\label{sec:order-polymers}

We define \emph{ordered polymers} to be connected subgraphs $\gamma$
of $G$ with an edge labelling
$\ell\colon E(\gamma) \to \{\text{occupied},\text{unoccupied}\}$,
subject to (i) $\abs{E_{u}(\gamma)}\leq \eta n$, where $E_u(\gamma)$
denotes the set of unoccupied edges of $\gamma$, and (ii)
$\cB_{\infty}(E_{u}(\gamma))=E(\gamma)$, i.e., the inductive boundary
procedure applied to the unoccupied edges of $\gamma$ stabilizes at
$\gamma$. Let $\cP_{\ord}=\cP_{\ord}(G)$ be the set of disordered
polymers in $G$.  As usual, two polymers are compatible if they are
vertex disjoint.

Let $c'(\gamma)$ denote the number of components of the graph
$(V,E\setminus E_{u}(\gamma))$ with fewer than $n/2$ vertices. We
think of these as `finite components', cf.\ Lemma~\ref{lem:simple}
below and also Section~\ref{sec:phase-transition} where we make a
similar definition with $G$ replaced by an infinite tree. The weight
function for ordered polymers is
\begin{equation*}
  w^{\ord}_\gamma \bydef q^{c'(\gamma)} (e^{\beta}-1)^{-|E_u(\gamma)|}.
\end{equation*}
The \emph{ordered polymer partition function} is
\begin{equation*}
\Xi^{\ord}\bydef
\sum_{\Gamma}\prod_{\gamma\in\Gamma}w^{\ord}_\gamma,
\end{equation*} 
where the sum is over compatible collections of ordered polymers.

We end this section with two lemmas that will be used to prove the
convergence of the ordered cluster expansion. First, recall the
following well-known fact about expander graphs, see e.g. \cite[Lemma
2.3]{trevisan2016lecture}.
\begin{lemma}
  \label{lem:simple}
 Let $G=(V,E)$ be a graph and 
 let $\abs{E'}\leq \eta \abs{E}$. Then $(V,E\setminus E')$
  contains a connected component of size
  $\left(1-\frac{\eta}{2\phi_{G}(1/2)}\right)|V|$.
\end{lemma}

Next we bound the number of unoccupied edges of a polymer in terms of $c'$.
\begin{lemma}
  \label{lem:cbound}
  For all ordered polymers $\gamma$,
  $|E_u(\gamma)| \ge \frac{5}{9}\Delta\cdot c'(\gamma)$.
\end{lemma}
\begin{proof}
  Let $S_{1},\dots, S_{t}$ denote the connected components of
  $(V,E\setminus E_{u}(\gamma))$.
  By  Lemma~\ref{lem:simple},
  since $\phi_G(1/2)\geq 1/10$,
  we may assume without loss of generality that
   $S_{1}$ contains at least
  $\left(1-{5\eta}\right)n\geq (1-\del n)$ vertices.  Let $U$ be the union
  of the vertices in $S_{2},\dots, S_{t}$, so
  $c'(\gamma)=t-1\le |U|$. Since any edge leaving $U$ must be
  unoccupied and $\abs{U}\leq \del n$, the
  claim follows since  $\phi_G( \del  ) \ge 5/9$. 
\end{proof}

\subsubsection{Convergence of Ordered Expansion}
\label{sec:conv-order-expans}

\begin{prop}
  \label{prop:OrdPoly}
  If $q=q(\Delta)$ is sufficiently large, then for all $\beta\in \mathbb C$ such that 
   $\left| e^{\beta}-1\right| \ge e^{\beta_{0}}-1$ the ordered polymer model satisfies~\eqref{eqKPcond} with 
  $r=\log q/(200\Delta)$.
\end{prop}
\begin{proof}
We will show that for  $\beta \ge \beta_0$ and for every $v \in V(G)$,
\begin{equation}
\label{eqKPgoalOrd}
\sum_{\gamma \ni v} e^{(1+r)|E(\gamma)|} \left|w^{\ord}_\gamma\right| \le  \frac{1}{ 2} \,.
\end{equation}
As in the proof of Proposition~\ref{prop:DisPoly}, this suffices to
verify condition~\eqref{eqKPcond}. 

We index polymers by $k = |E_u(\gamma)|$.  By Lemma~\ref{lem:ord-gen}
$|E(\gamma)| \le 10 |E_u(\gamma)|$.  By~\cite[Lemma~2.1
(c)]{BorgsChayesKahnLovasz} we can bound the number of polymers with
$|E_u(\gamma)|=k$ containing a vertex $v$ by $(2e\Delta)^{10k}$, where
the factor of $2$ accounts for the choice of occupied/unoccupied for
each edge.  Then, since
$\abs{e^{\beta} -1} \ge q^{\frac{1.9}{\Delta}}$, by Lemma~\ref{lem:cbound}
\begin{align*}
  \sum_{\gamma \ni v}  e^{(1+r)| E(\gamma)|} \left|w^{\ord}_\gamma\right| 
  &\le  \sum_{k \ge 1}  (2e^{2+r}\Delta)^{10k}   q^{\frac{9k}{5\Delta}} \abs{e^{\beta} -1}^{-k} \\ 
&\le  \sum_{k \ge 1} (2e^{2+r}\Delta)^{10k}   q^{ \left ( \frac{9}{5\Delta}  - \frac{1.9}{\Delta}     \right ) k}  \\
&= \sum_{k \ge 1} (2e^{2+r}\Delta)^{10k}   q^{ - \frac{ k}{10 \Delta}} , 
\end{align*}
which is at most $1/ 2$ for $q=q(\Delta)$ sufficiently large.
\end{proof}

\subsubsection{Ordered polymer model measure on edges}
\label{sec:ord-on-edge}
Let $\nu_{\ord}$ be the polymer model measure defined
by~\eqref{eqNuGamma}. As for the disordered polymer model we can
define a measure $\overline \nu_{\ord}$ on
$\Omega=\{0,1\}^{E}$. 
To obtain a sample $A$ from $\overline \nu_{\ord}$ we sample a
collection $\Gamma$ of compatible ordered polymers according to
$\nu_{\ord}$ and then let
\begin{equation*}
A = E \setminus \bigcup_{\gamma \in \Gamma} E_u(\gamma)  \, .
\end{equation*}

\subsection{Consequences of the cluster expansion convergence}
\label{secConseq}

This section derives some consequences of Theorem~\ref{thmKP}.  Let
$\cC_{\dis}$ and $\cC_{\ord}$ be the sets of clusters of polymers in
$\cP_{\dis}$ and $\cP_{\ord}$, respectively.  For
$\ast\in \{\dis,\ord\}$ we write $\cC_{\ast}^{v}$ for the set of
clusters containing the vertex $v$.  We will always assume that $q$ is
large enough that Propositions~\ref{prop:DisPoly}
and~\ref{prop:OrdPoly} apply.

\subsubsection{Truncated cluster expansion error bounds}
\label{sec:trunc-clust-expans}

The following lemma will be used extensively
in
Section~\ref{sec:finite-size-scaling} and Section~\ref{sec:phase-transition}.
\begin{lemma}
\label{lemVertexTail}
For every $v \in G$ and $m\geq 1$,
\begin{equation}
\label{eqVertexTail}
\sum_{\substack{\Gamma\in \cC^{v}_{\ast} \\ \|\Gamma\| \ge m}}\left |  w^{\ast}(\Gamma) \right| \le q^{-\frac{m}{200 \Delta}} \,.
\end{equation}
\end{lemma} 
\begin{proof}
  This follows from Propositions~\ref{prop:DisPoly}
  and~\ref{prop:OrdPoly} by applying~\eqref{eqKPtail} with $\gamma$ a single edge
  containing $v$ and $r=\log q/(200\Delta)$.
\end{proof}

An important consequence of
Lemma~\ref{lemVertexTail} is the following. For $m\geq 1$, define
truncated cluster expansions
\begin{equation}\label{eq:Tmdef}
  T_{m}^{\ast} \bydef \mathop{\sum_{\Gamma\in\cC_{\ast}}}_{\|\Gamma\|< m}
  w^{\ast}(\Gamma), \qquad \ast\in \{\dis,\ord\}
\end{equation}
By summing~\eqref{eqVertexTail} over all vertices we have that
\begin{equation}
  \label{eq:EdgePolySum}
 \sum_{\substack{\Gamma\in \cC_{\ast} \\ \|\Gamma\| \ge m}}\left |  w^{\ast}(\Gamma) \right| \le nq^{-\frac{m}{200 \Delta}}, \qquad \ast \in \{\dis,\ord\} 
\end{equation}
and so
\begin{equation}
  \label{eq:approx}
  \abs{T_{m}^{\ast} - \log \Xi_{\ast}}\leq n q^{-\frac{m}{200 \Delta}}, \qquad \ast \in \{\dis,\ord\}\, ,
\end{equation}
which will be used in Section~\ref{sec:algor-cons}.

\subsubsection{Probabilistic estimates for measures on edges}
\label{sec:prob-estim-meas}

We first show that the set of edges contained in a polymer sample is
typically small.
Recall that for $\ast\in\{\dis, \ord\}$,
 $\nu_{\ast}$ is the Gibbs measure associated to the $\ast$-polymer model 
(given by~\eqref{eqNuGamma}) and $\overline \nu_{\ast}$ is the measure
on $\{0,1\}^E$ induced by $\nu_{\ast}$ (as defined in Subsections~\ref{sec:disord-polym-meas}, \ref{sec:ord-on-edge}).
\begin{lemma}
  \label{lem:largedev}
  Let $\ast\in\{\dis, \ord\}$ and let $\mathbf \Gamma$ be a random
  configuration of compatible polymers sampled from $ \nu_{\ast}$.
  Then for $\alpha>0$,
\[
\P(\|\mathbf \Gamma\|>\alpha n)\leq e^{-2n}\, 
\]
for $q=q(\Delta, \alpha)$ sufficiently large. 
\end{lemma}
\begin{proof}
  Fix $\ast\in\{\dis, \ord\}$ and let $\mathbf \Gamma$ be a random
  sample from $ \nu_{\ast}$.  Consider the cumulant generating
  function
  \begin{equation*}
    h_t(\|\mathbf \Gamma\|)\bydef \log \E e^{t\|\mathbf \Gamma\|}\, .
  \end{equation*}
  For $t>0$, we introduce an auxiliary polymer model on $\cP_{\ast}$
  by modifying the polymer weights:
 \begin{equation*}
   \tilde w^{\ast}_{\gamma}=w^{\ast}_{\gamma}\cdot e^{t|E(\gamma)|}\, .
 \end{equation*}
 Let $\tilde \Xi_{\ast}$ be the modified partition function.  Then
 $\log \tilde \Xi_{\ast}- \log \Xi_{\ast}= h_t(\|\mathbf \Gamma\|)$.
 Setting $t=3/\alpha$, the estimates used to prove
 Propositions~\ref{prop:DisPoly} and~\ref{prop:OrdPoly} still hold
 with $w^{\ast}$ replaced with $\tilde w^{\ast}$ if $q$ is
   sufficiently large, and so, by~\eqref{eq:EdgePolySum} with $m=1$ and with the modified weights, we have
\begin{equation*}
\log \tilde\Xi_{\ast}\leq \sum_{\gamma\in \cC_{\ast}}|\tilde w^{\ast}(\Gamma)|\leq n \, .
\end{equation*}
Since $\Xi_{\ast}\geq 1$ we then have $h_t(\|\mathbf \Gamma\|)\leq n$ also.
By Markov's inequality
\begin{equation*}
\P(\|\mathbf \Gamma\|> \alpha n)\leq e^{-t\alpha n} \E e^{{t}\|\mathbf \Gamma\|}\leq e^{-t\alpha n} e^{n}\le e^{-2n}\, . \qedhere
\end{equation*}
\end{proof}

The next lemma shows that $\overline \nu_{\dis}$ and $\overline \nu_{\ord}$ exhibit exponential decay of correlations. In the following we let $\mathbf A $ denote a random edge subset drawn according to the measure $\overline \nu_{\dis}$ or $\overline \nu_{\ord}$.
\begin{lemma}
\label{lemPolyExpDecay}
\begin{enumerate}
\item For $\beta \le \beta_1$, $\overline \nu_{\dis}$ exhibits exponential decay of correlations and $|\mathbf A|$ obeys a central limit theorem with respect to $\overline \nu_{\dis}$.   

\item For $\beta \ge \beta_0$, $\overline \nu_{\ord}$ exhibits exponential decay of correlations and $|\mathbf A|$ obeys a central limit theorem with respect to $\overline \nu_{\ord}$.  
\end{enumerate}
\end{lemma}
\begin{proof}
These are  standard consequences of the condition~\eqref{eqKPcond}, so
we will only provide a sketch.  To prove exponential decay of correlations, we need to show that there exist constants $C, \eps>0$ so that for all $e, f \in E$,
\begin{equation}
\label{eq:corneeded}
\left |\overline \nu_{\dis} (e,f)  -  \overline \nu_{\dis} (e)
  \overline \nu_{\dis} (f)   \right | \le C e^{-\eps \dist(e,f)} \,. 
\end{equation}
 See
e.g.~\cite[Theorem 1.3]{cannon2019bipartite} for details.
Establishing~\eqref{eq:corneeded} amounts
to observing that the correlation between edges $e$ and $f$,
$\overline \nu_{\dis} (e,f) - \overline \nu_{\dis} (e) \overline
\nu_{\dis} (f)$, equals a weighted sum over clusters of disordered
polymers containing both $e$ and $f$ of the cluster weight.  The size
of any such cluster is at least $\dist(e,f)$, and a
tail bound like~\eqref{eq:approx} shows the total weight of
these clusters is exponentially small in $\dist(e,f)$.

Likewise for $\overline \nu_{\ord}$, by taking complements and using
inclusion-exclusion it is enough to show that
\begin{equation}
\label{eq:corneeded2}
\left |\overline \nu_{\ord} (\overline e, \overline f)  -  \overline
  \nu_{\ord} (\overline e) \overline \nu_{\ord} ( \overline f)
\right | \le C e^{-\eps \dist(e,f)} \,,
\end{equation}
where $ \overline \nu_{\ord} (\overline e) $ is the probability
$e \notin A$ and $\overline \nu_{\ord} (\overline e, \overline f) $ is
the probability that $\{e \notin A\} \wedge \{f \notin A\}$.
Again~\eqref{eq:corneeded2} can be expressed as a sum of cluster
weights with clusters of size at least $\dist(e,f)$ and so we obtain
exponential decay of correlations.

To prove a central limit theorem for $|\mathbf A|$ under $\overline \nu_{\dis}$ we first  center and normalize, letting $Y = ( |\mathbf A| - \E |\mathbf A| )/\sigma $ where $\sigma^2 =\var(|\mathbf A|)$.  Now by the method of moments (or cumulants) it is enough to show that for each fixed $k \ge 3$, the $k$th cumulant of $Y$, $\kappa_k(Y)$, vanishes as $n \to \infty$.    Using the cluster expansion we can express
\begin{align*}
|\kappa_k(Y) | &=   \left |\sum_{\Gamma \in \cC_{\dis}}   w^{\dis}(\Gamma) \left (\frac{\|\Gamma\|}{\sigma} \right)^k \right | \\
&\le \frac{1}{\sigma^3}   \sum_{\Gamma \in \cC_{\dis}}   |w^{\dis}(\Gamma)| \| \Gamma\|^k  \\
&\le \frac{ \Delta  n}{\sigma^3} \sum_{t \ge 1} e^{-kt} t^k = O \left( \frac{n}{\sigma^3}  \right)\,,
\end{align*}
where we applied~\eqref{eqKPtail} in the last line. 
A simple conditioning argument (see e.g.~\cite[Lemma 9]{davies2021approximately}) shows that $\sigma = \Omega(n^{1/2})$, and so for $k\ge 3$, $\kappa_k(Y) \to 0$, as desired. 
 The proof for $\overline \nu_{\ord}$ is similar, substituting missing edges for occupied edges. 
\end{proof}

\begin{rmk}
  The correlation between edges $e$ and $f$ is a joint cumulant of the
  indicator random variables that each is in $A$.  The same techniques
  can be used to show that joint cumulants of the indicators of $k$ edges decay
  exponentially in the size of the minimum spanning tree connecting
  the edges in $G$, see~\cite{cannon2019bipartite} for more details.
\end{rmk}

The next lemma states that up to total variation distance $\eps$, the
measures induced by $\overline \nu_{\dis}$ and $\overline \nu_{\ord}$
on the local neighborhood of any vertex $v$ are determined by clusters
contained in a (larger) neighborhood of $v$. To make this precise, for
$\ast\in \{\dis,\ord\}$, $T\in \N$ and a vertex $v$, define
$\overline\nu_{\ast}^{B_{T}(v)}$ to be the projection of
$\overline\nu_{\ast}$ to $\{0,1\}^{E(B_{T}(v))}\subset
\{0,1\}^{E}$. Here $B_{T}(v)$ is the ball of radius $T$ around $v$.
\begin{lemma}
\label{lemDepthDistribution}
Suppose $T>0, \eps >0$, and $\ast\in \{\dis,\ord\}$. There is an $m$
large enough as a function of $\Delta, T, \eps$ so that for any
$v\in V$, $\overline \nu_{\ast}^{B_{T}(v)}$ is determined up to total
variation distance $\eps$ by clusters which lie entirely in $B_m(v)$.
\end{lemma}
\begin{proof}
  We prove this for $\overline \nu_{\dis}$, the proof for $\overline
  \nu_{\ord}$ is identical. Let $A\subseteq
  E$ be distributed according to
  $\overline\nu_{\dis}$. We first claim it is enough to give, for each
  $F \subseteq E(B_T(v))$, a quantity $\kappa(F)$ so that
  \begin{enumerate}
  \item $| \kappa(F) - \overline \nu_{\dis}( A\subset F^c)| 
    \le \eps 2^{-2 \Delta^T}$. 
  \item $\kappa(F)$ depends only on clusters contained in
    $B_m(v)$ for some $m= m(T,\Delta,\eps)$.
  \end{enumerate}
The lemma follows from these two properties by calculating $\overline
\nu_{\ast}^{B_{T}(v)}(\cdot)$ 
via inclusion-exclusion, and summing the error bound over all
subsets of $E(B_T(v))$.

To find such a $\kappa(F)$, observe there is an exact formula for
$\overline \nu_{\dis}( A\subset F^{c})$ in terms of clusters:
\begin{equation*}
\log \overline \nu_{\dis}( A\subset F^{c}) 
= - \sum_{\substack { \Gamma \in \cC_{\dis}}}  
 w^{\dis}(\Gamma) 1_{\Gamma\cap F\ne\emptyset} \,,
\end{equation*}
where $1_{\Gamma\cap F\ne\emptyset}$
indicates that $\Gamma$ contains a polymer which contains an edge from $F$.
Since $F$ spans at most $|V(B_T(v))|\leq \Delta^T+1$ vertices, by Lemma~\ref{lemVertexTail} we can truncate the RHS and
obtain
\begin{equation}
  \label{eq:mult}
\left| \log \overline \nu_{\dis}( A\subset F^{c})
+ \sum_{\substack { \Gamma \in \cC_{\dis}  
  \\ \|\Gamma \| \le m}  } w^{\dis}(\Gamma)1_{\Gamma \cap F \ne \emptyset} \right | \le q^{- \frac{ m}{200 \Delta}}(\Delta^T+1) \,.
\end{equation}
Note the quality of~\eqref{eq:mult} is independent of $F$. The desired quantity is
\begin{align*}
\kappa(F) &= \exp \left ( -\sum_{\substack { \Gamma \in \cC_{\dis}  
  \\ \|\Gamma \| \le m}  } w^{\dis}(\Gamma)1_{\Gamma \cap F \ne \emptyset}   \right)  \,.
\end{align*}
Taking $m$ large enough as a function of $\Delta, T,
\eps$ gives properties (1) and (2) and proves the lemma, as, since
$\kappa(F)\in
(0,1)$, the accurate multiplicative approximation guaranteed
by~\eqref{eq:mult} implies an accurate additive approximation.
\end{proof}

\subsection{Polymer model approximation of the partition function}
\label{sec:accur-order-expans}
Using the results above we now show that the scaled polymer model
partition functions are good approximations to $Z^{\dis}$ and
$Z^{\ord}$. We also show that the measures $\overline \nu_{\dis}$
and $\overline \nu_{\ord}$ on edge sets are good approximations to
$\mu_{\dis}$ and $\mu_{\ord}$, where $\mu_\dis$ and $\mu_\ord$ are the
random cluster measure $\mu$ conditioned
on $\Omega_\dis$ and $\Omega_\ord$, respectively.

\begin{lemma}
  \label{lem:DisQual}
  For $q=q(\Delta, \delta)$ sufficiently large and $\beta \le \beta_1$,
  \begin{equation}\label{eq:DisQual}
    \abs{Z^{\dis} - q^{n}\cdot \Xi^{\dis}} \leq e^{-n} Z^{\dis}\, .
  \end{equation}
  Moreover,
  \begin{equation}
  \label{eqTVdistDIS}
\| \mu_{\dis} - \overline\nu_{\dis} \| _{TV} \le e^{-n} \,. 
\end{equation}
\end{lemma}
\begin{proof}
  Write $\Xi^{\dis}_{\leq }$ for the contribution to $\Xi^{\dis}$ of
  compatible collections $\Gamma$ of polymers with
  $|E(\Gamma)|\leq \eta |E|$ and let
  $\Xi^{\dis}_{>} = \Xi^{\dis} -\Xi^{\dis}_{\leq }$. By definition we
  have $Z^{\dis}=q^{n}\cdot \Xi^{\dis}_{\leq }$. By Lemma~\ref{lem:largedev} we have
  $\Xi^{\dis}_{>}\leq e^{-2n}\Xi^{\dis}\leq 2 e^{-2n} \Xi^{\dis}_{\leq
  }$.  It follows that
  \begin{equation}
    \abs{Z^{\dis} - q^{n}\cdot \Xi^{\dis}}= q^n \Xi^{\dis}_{> }\leq
    2q^n e^{-2n}\cdot \Xi^{\dis}_{\leq }= e^{-n} Z^{\dis}\, .
  \end{equation}
  The proof of~\eqref{eqTVdistDIS} is nearly identical, see, e.g., the
  proof of~\cite[Lemma 14]{jenssen2019revisited}.
\end{proof}

We now turn to $Z^{\ord}$. This requires a preparatory lemma.
\begin{lemma}
  \label{lem:factor}
  Let $\Gamma = \{\gamma_1, \ldots, \gamma_k\}$ be a collection of
  compatible polymers, and assume $|E_{u}(\Gamma)|\leq \eta |E|$. The
  number of connected components in the graph
  $G-(E_u(\gamma_1)\cup\ldots \cup E_u(\gamma_k))$ is
  $c'(\gamma_1)+\ldots+c'(\gamma_k)+1$.
\end{lemma}
\begin{proof}
  Let $S_1,\ldots, S_t$ denote the vertex sets of the components of
  $G\setminus E_{u}(\Gamma)$, 
  and without loss of generality let
  $S_1$ be the largest component.
  As in the proof of Lemma~\ref{lem:cbound},
  we have that
 $|S_{1}|\geq
  (1-\del)n$ by Lemma~\ref{lem:simple} and our assumption on
  $|E_{u}(\Gamma)|$.  We claim that for for each
  $i\geq2$, we have $E(S_i, S_i^{c})\subset
  E_u(\gamma_j)$ for some $j$.  This suffices to prove the lemma.
 
  Suppose, towards a contradiction, that $E(S_2, S_2^{c})$ contains edges from
  more than one of the sets $E_u(\gamma_i)$. Without loss of
  generality, let the $\ell>1$ indices for which $E(S_2,S_2^c)\cap
  E_u(\gamma_i)\neq\emptyset$ be $1,2,\dots, \ell$. 
  
  Let $T\subset S_2$ denote the set of vertices in $S_2$ that have
  fewer than $\Delta$ incident edges from each of the sets
  $E_u(\gamma_1), \ldots, E_u(\gamma_\ell)$.  If $T=\emptyset$,
  pairwise compatibility of polymers implies each vertex in $S_{2}$ is
  incident to $\Delta$ edges in exactly one of the sets
  $E_{u}(\gamma_{i})$, $1\leq i\leq\ell$. Since $S_{2}$ is connected,
  all edges incident to $S_{2}$ must in fact be from the \emph{same}
  set $E_{u}(\gamma_{i})$, contradicting $\ell>1$.

  To conclude the proof, we show $T=\emptyset$ is the only
  possibility. Suppose not, i.e., $T\neq\emptyset$. Note that all of
  the edges in $E(T, T^{c})$ belong to
  $E(\gamma_1)\cup\ldots \cup E(\gamma_\ell)$. Moreover
  $|T|\leq \del n$ and so $|E(T,T^{c})|\geq 5\Delta |T|/9$ since
  $\phi_G(\del)\geq5/9$.  It follows that there is a vertex $u\in T$
  incident to $\geq 5\Delta /9$ edges in
  $E(\gamma_1)\cup\ldots \cup
  E(\gamma_\ell)$.   Without loss of generality, it must be the case that $u$ is incident to
  $\geq 5\Delta /9$ edges in
  $E(\gamma_1)$, as $u$ cannot be an endpoint of unoccupied edges in
  distinct compatible polymers.  By the definition of the polymer $\gamma_1$, all of
  the $\Delta$ edges incident to $u$ must then belong to $E(\gamma_1)$,
  contradicting the definition of $T$. 
 \end{proof}

\begin{lemma}
  \label{lem:OrdQual}
 If $q=q(\Delta, \delta)$ is sufficiently large and
  $\beta \ge \beta_0$, then
  \begin{equation}
    \label{eq:OrdTail}
    \abs{Z^{\ord} - q(e^{\beta}-1)^{\frac{\Delta n}{2}}\cdot
      \Xi^{\ord}}\leq e^{- n} Z^{\ord}.
  \end{equation}
  
   Moreover,
  \begin{equation}
  \label{eqTVdistORD}
  \| \mu_{\ord} - \overline\nu_{\ord} \| _{TV} \le e^{-n} \,. 
\end{equation}
\end{lemma}
\begin{proof}
  Write $\Xi^{\ord}_{\leq }$ for the contribution to $\Xi^{\ord}$ of
  compatible collections $\Gamma$ of polymers with
  $|E_{u}(\Gamma)|\leq \eta |E|$. Then
  $Z^{\ord}=q(e^{\beta}-1)^{\frac{\Delta n}{2}}\Xi^{\ord}_{\leq}$ by
  Lemma~\ref{lem:factor}.  By Lemma~\ref{lem:largedev} we have
  $\Xi^{\ord}_{>}\leq e^{-2n}\Xi^{\ord}\leq 2 e^{-2n} \Xi^{\ord}_{\leq
  }$.  It follows that
  \begin{equation*}
    \abs{Z^{\ord}- q(e^{\beta}-1)^{\frac{\Delta n}{2}}\Xi^{\ord}}
    =q(e^{\beta}-1)^{\frac{\Delta n}{2}}\Xi^{\ord}_{>}
     \leq
    2e^{-2n}q(e^{\beta}-1)^{\frac{\Delta
        n}{2}}\Xi^{\ord}_{\leq } 
        \le 
        e^{-n}Z^{\ord}.
  \end{equation*}
As in the proof of Lemma~\ref{lem:DisQual}, \eqref{eqTVdistORD} follows from a similar argument. 
\end{proof}

\begin{cor}
  \label{cor:polyapprox}
  If $q$ and $n$ are sufficiently large as a function of $\Delta$ and $\delta$, then for all $\beta>0$
 \begin{align} \label{eqtildeZ}
 \tilde Z(q,\beta) \bydef q^n \cdot \Xi^{\dis} +
   q(e^{\beta}-1)^{\Delta n/2} \cdot \Xi^{\ord}
 \end{align} 
 is an $e^{-n/2}$-relative approximation to $Z_G(q,\beta)$. 
\end{cor}
\begin{proof}
  The result follows by applying Lemmas~\ref{lem:ErrQual},
  \ref{lem:DisQual} and~\ref{lem:OrdQual} in the the ranges
  $\beta \le \beta_0$, $ \beta_0 \le \beta \le\beta_1$ and
  $\beta \ge\beta_1$\, .
\end{proof}

\section{Phase transitions on trees and random graphs}
\label{sec:phase-transition}

In this section we prove that the $q$-state random cluster model on
locally tree-like graphs in $\cG_{\Delta,\del}$ has a unique phase
transition, and we characterize the critical point. To do this we
define ordered and disordered polymer models for the random cluster
model on the infinite $\Delta$-regular tree $\bT_\Delta$ and its
finite depth-$L$ truncations $\bT_\Delta^L$, and then relate these models
to the random cluster model on locally tree-like $\Delta$-regular
graphs.

\subsection{The random cluster model and polymer models on finite trees}
\label{secPolyonTrees}

Fix $\Delta$ and let $\bT_{\Delta}$ be the rooted infinite
$\Delta$-regular tree with root vertex $r$.  Let $\bT_\Delta^L$ denote
the finite subtree of $\bT_{\Delta}$ with root $r$ and depth $L$ where
each non-leaf vertex has degree $\Delta$.

\subsubsection{Free boundary conditions and the disordered polymer
  model on finite trees}\label{secDisPolyonTrees}
We start with $\bT_\Delta^L$.  The random cluster model on
$\bT_\Delta^L$ with free boundary conditions is simply the random
cluster model on $\bT_\Delta^L$.  In particular,
\begin{equation*}
Z^{\free}_{\bT_\Delta^L}(q,\beta) = \sum_{A \subseteq E(\bT_\Delta^L)}  q^{c(A)} (e^{\beta}-1)^{|A|} \,,
\end{equation*}
where $c(A)$ is the number of connected components of
$(V(\bT_\Delta^L),A)$. To distinguish these boundary conditions we
call the resulting measure on edges the \emph{free random cluster
  model on $\bT_{\Delta}^{L}$}. Recall from
Section~\ref{sec:locally-tree-like} that $\mu^{\free}$ is the weak
limit of the free random cluster model measures on $\bT_{\Delta}^{L}$ as $L\to\infty$.

Next we define a disordered polymer model on $\bT_{\Delta}^{L}$.  Define polymers to
be connected components of $\bT_\Delta^L$.  Unlike in
Section~\ref{sec:polym-model-repr}, we do not restrict the size of
polymers. The weight of a polymer is again given by
$w^{\dis}_\gamma = q^{1 - |\gamma|} (e^{\beta} -1)^{|E(\gamma)|}$.
Call the resulting polymer model partition function
$\Xi^{\dis}_{\bT_\Delta^L}$.
  
\begin{lemma}
  \label{lemFreePartition}
  We have the equality
  \begin{equation*}
    Z^{\free}_{\bT_\Delta^L}(q,\beta) = q^{|V(\bT_\Delta^L)|}   \cdot \Xi^{\dis}_{\bT_\Delta^L} \,.
  \end{equation*}
  Moreover, the induced measure $\overline\nu_{\dis}$ on edges is
  the free random cluster measure on $\bT_{\Delta}^{L}$. 
  \end{lemma}
\begin{proof}
  This follows from two facts: 1) There is a bijection between subsets
  $A \subseteq E(\bT_\Delta^L)$ and collections of mutually compatible
  polymers and 2) if $A \subseteq E(\bT_\Delta^L)$ corresponds to the
  set of polymers $\{\gamma_1, \ldots, \gamma_k\}$, then
  $c(A)=|V(\bT_\Delta^L)|+\sum_i (1-|\gamma_i|)$. That is, the
  bijection of 1) is weight-preserving.
\end{proof}

\subsubsection{Wired boundary conditions and the ordered polymer model
  on finite trees}

Informally, the random cluster model on $\bT_\Delta^L$ with wired
boundary conditions is obtained by declaring that the boundary
vertices belong to a single connected component.  Formally, we call a
subset of $V(\bT_\Delta^L)$ \emph{finite} if it contains no boundary
vertex. Then,
\begin{equation*}
Z^{\wire}_{\bT_\Delta^L}(q,\beta) \bydef \sum_{A \subseteq E(\bT_\Delta^L)}  q^{c_{w}(A)} (e^{\beta}-1)^{|A|} \,,
\end{equation*}
where $c_{w}(A)$ is the number of finite connected components of
$(V(\bT_\Delta^L), A)$ plus 1; the additional one is to account for
the single component containing the boundary vertices. We call the
resulting measure on edges the \emph{wired random cluster measure on
  $\bT_{\Delta}^{L}$}. Recall from Section~\ref{sec:locally-tree-like}
that $\mu^{\wire}$ is the weak limit of the wired random cluster
measures on $\bT_{\Delta}^{L}$ as $L\to\infty$.

Ordered polymers on $\bT_\Delta^L$ are defined as in
Section~\ref{sec:polym-model-repr}, but with no restriction on their size.
The weight function is
\begin{equation*}
  w^{\ord}_\gamma \bydef q^{c^{\prime}(\gamma)} (e^{\beta}-1)^{-|E_u(\gamma)|},
\end{equation*}
where $c^{\prime}(\gamma)$ is the number of finite connected
components of the graph
$(V(\bT_\Delta^L),E(\bT_\Delta^L) \setminus E_u(\gamma))$. We have the
following analogue of Lemma~\ref{lem:factor}.
\begin{lemma}
  \label{lem:factorTree}
  Let $\Gamma = \{\gamma_1, \ldots, \gamma_k\}$ be a collection of
  compatible ordered polymers on $\bT_\Delta^L$. Then
  \begin{equation}
    \label{eq:cwfact}
c_{w} \left( E(\bT_\Delta^L) \setminus \bigcup_{i=1}^k E_u(\gamma_i) \right) =1+ \sum_{i=1}^kc^{\prime}(\gamma_i) \,.
\end{equation}
\end{lemma}
\begin{proof}
Let $S_1,\ldots, S_t$ denote the vertex sets of the finite components of 
$E(\bT_\Delta^L) \setminus \bigcup_{i=1}^k E_u(\gamma_i)$.
In particular, by the definition of $c_{w}$, 
the left hand side of~\eqref{eq:cwfact} is equal to $t+1$.

If $S\subseteq V(\bT_\Delta^L)$ is finite,
then $|E(S, S^c)|\geq (\Delta-2)|S|$. 
It then follows, as in the proof of Lemma~\ref{lem:factor},
that each of the sets $S_1,\ldots, S_t$ is 
incident to edges from precisely one 
of the polymers $\gamma_1, \ldots, \gamma_k$. 
The result follows.
\end{proof}

With this we prove that the polymer model partition function equals the wired random cluster partition function after scaling.

\begin{lemma}
  \label{lemWirePartition}
  We have the equality
  \begin{equation*}
    Z^{\wire}_{\bT_\Delta^L}(q,\beta) = q (e^{\beta}-1)^{|E(\bT_\Delta^L)|} \Xi^{\ord}_{\bT_\Delta^L} \,.
  \end{equation*}
  Moreover, the induced measure $\overline\nu_{ord}$ on edges is the
  wired random cluster measure on $\bT_{\Delta}^{L}$.
\end{lemma}
\begin{proof}
This follows from Lemma~\ref{lem:factorTree}, which implies there is a
weight-preserving bijection between sets $A\subseteq E(\bT_\Delta^L)$ 
and collections of mutually compatible ordered polymers. 
\end{proof}

\subsection{Infinite trees and limiting free energies}
\label{sec:limit-free-energ}

To motivate the definitions that follow, we begin by rewriting the cluster
expansions for $\Xi^{\ord}$ and $\Xi^{\dis}$ for a given
finite $\Delta$-regular graph $G$ on $n$ vertices.   For a cluster $\Gamma$, let $u(\Gamma)$ be the number of distinct vertices contained
in $\Gamma$. Write $\cC_{\dis}^{v}(G)$ for the set of disordered clusters
containing $v$, and similarly for $\ord$.  Then 
\begin{align}
\label{eqPinnedDis}
\log \Xi^{\dis} &= \sum_{v \in V(G)}  \sum_{\Gamma\in\cC_{\dis}^{v}(G)}  \frac{1}{u(\Gamma)}w^{\dis}(\Gamma), \\
\label{eqPinnedOrd}
\log \Xi^{\ord} &=  \sum_{v \in V(G)} \sum_{\Gamma\in\cC_{\ord}^{v}(G)} \frac{1}{u(\Gamma)} w^{\ord}(\Gamma) \,.
\end{align}

Using this as a model, we consider ordered and disordered polymer
models on the infinite $\Delta$-regular tree $\bT_\Delta$ rooted at
$r$.  Here we define polymers and weights exactly as for
$\bT_\Delta^L$ above in Section~\ref{secPolyonTrees}, but with the
additional condition that the polymers be finite.  In particular, for
an ordered polymer $\gamma$ the weight function is
\begin{equation*}
 w^{\ord}_\gamma \bydef q^{c^{\prime}(\gamma)} (e^{\beta}-1)^{-|E_u(\gamma)|},
\end{equation*}
where $c^{\prime}(\gamma)$ is the number of finite connected
components of the graph
$(V(\bT_\Delta),E(\bT_\Delta) \setminus E_u(\gamma))$; here we
mean finite in the usual sense of finite cardinality.

\begin{lemma}
\label{lemTreeConverge}
For $q=q(\Delta)$ sufficiently large the following hold with
$r=\log q/(200\Delta)$:
\begin{enumerate}
\item For $\beta\in \mathbb C$ such that 
   $\left| e^{\beta}-1\right| \le e^{\beta_{1}}-1$,  
the disordered polymer model on $\bT_\Delta$
satisfies~\eqref{eqKPcond}.
\item For $\beta\in \mathbb C$ such that 
   $\left| e^{\beta}-1\right| \ge e^{\beta_{0}}-1$,  
the ordered polymer model on $\bT_\Delta$
satisfies~\eqref{eqKPcond}.
\end{enumerate}
\end{lemma}
\begin{proof}
We can mimic the proofs of Propositions~\ref{prop:DisPoly} and \ref{prop:OrdPoly} once we note that the tree $\bT_\Delta$ satisfies the following optimal expansion condition: for a finite set $S\subset V(\bT_\Delta)$, $|E(S, S^c)|\geq (\Delta-2)|S|$.
\end{proof}

Let $\cC_{\ord}^{r}$ and $\cC_{\dis}^{r}$ be the respective sets of clusters on
$\bT_\Delta$ containing the root $r$.  Define
\begin{equation*}
  \overline f_{\ord}(\beta,q) \bydef  \sum_{\Gamma \in \cC_{\ord}^r}
                                \frac{1}{u(\Gamma)} w^{\ord}(\Gamma),
                                \qquad 
 \overline f_{\dis} (\beta, q) \bydef  \sum_{\Gamma \in \cC_{\dis}^r} \frac{1}{u(\Gamma)} w^{\dis}(\Gamma)  \,,
\end{equation*}
By Lemma~\ref{lemTreeConverge} and~\eqref{eqKPtail} of Theorem~\ref{thmKP},
for $q$ large and $\beta\in[\beta_0, \beta_1]$ these series converge and are functions of $\Delta, \beta, q$.
Further define
\begin{equation}
  \label{eq:ftreelim}
  f_{\ord}(\beta,q) \bydef \frac{\Delta}{2} \log(e^{\beta}-1) +
  \overline   f_{\ord}, \qquad
  f_{\dis} (\beta, q)  \bydef \log q + \overline f _{\dis}.
\end{equation}

\begin{prop}
  \label{prop:limf}
  For any sequence of locally tree-like graphs
  $G_{n}\in \cc G_{\Delta,\delta}$ and $q = q(\Delta,\delta)$ large
  enough, we have
  \begin{align}
    \label{eq:flimdis}
    \lim_{n \to \infty} &\frac{1}{n} \log Z^{\dis}_{G_n} 
    = f_\dis, && \abs{e^{\beta}-1} \le e^{\beta_{1}}-1\\
    \label{eq:flimord}
    \lim_{n \to \infty} &\frac{1}{n} \log Z^{\ord}_{G_n} 
    = f_\ord, &&  \abs{e^{\beta}-1}\ge e^{\beta_{0}} -1.
  \end{align}
  Moreover these limits are uniform on the given regions of $\beta$.
\end{prop}
\begin{proof}
  We give the proof of the first statement, as the proof of the second
  statement is the same up to changes in notation. Fix $\eps>0$ and
  let $m= \log (2/\eps)$.  Let $V_m$ be the set of vertices of $G_n$
  whose depth-$m$ neighborhood is not a tree.  Recall that we write $\cC_{\dis}^{v}(G_n)$ for the set of clusters on
  $G_{n}$ that contain a vertex $v$. Then using Lemma~\ref{lem:DisQual},
\begin{align*}
\left | \log Z^{\dis}_{G_n} - n f _{\dis} \right| 
&\leq \left | \log (q^n\Xi^{\dis}_{G_n}) - n f _{\dis} \right| +2e^{-n}  \\
&=  \left |\sum_{v \in V}    \left( \sum_{\Gamma\in\cC_{\dis}^{v}(G_n)}  \frac{1}{u(\Gamma)}w^{\dis}(\Gamma)   -   \sum_{\Gamma \in \cC_{\dis}^{r}} \frac{1}{u(\Gamma)} w^{\dis}(\Gamma)\right )  \right|  + 2e^{-n}\\
&\le \sum_{v \in V}   \left |    \sum_{\Gamma\in\cC_{\dis}^{v}(G_n)}  \frac{1}{u(\Gamma)}w^{\dis}(\Gamma)   -   \sum_{\Gamma \in \cC_{\dis}^{r}} \frac{1}{u(\Gamma)} w^{\dis}(\Gamma)  \right| + 2e^{-n} \\
&\le \eps n + \sum_{v \in V}   \left |
  \sum_{\substack{\Gamma\in\cC_{\dis}^{v}(G_n) \\ \|\Gamma\| \le m}}  \frac{1}{u(\Gamma)}w^{\dis}(\Gamma)   -   \sum_{\substack{\Gamma \in \cC_{\dis}^{r} \\ \|\Gamma\| \le m}} \frac{1}{u(\Gamma)} w^{\dis}(\Gamma)  \right| +2e^{-n}\\
&\le \eps n  + 2 |V_{m}|+2e^{-n}\,.
\end{align*}
The penultimate inequality follows from Lemma~\ref{lemVertexTail} applied to each series inside the absolute values.  The final inequality follows from applying Lemma~\ref{lemVertexTail} to $v \in V_m$ to bound each series by $1$ in absolute value and noting that if $v \notin V_m$, the two series inside the absolute values are identical.   Since $G_n$ is locally tree-like, $|V_{m}|/n \to 0$ and so $\limsup |    \frac{1}{n} \log Z^{\dis}_{G_n} - f_\dis|  \le \eps$.  Taking $\eps \to 0$ proves the statement. 

The same proof shows that $\lim_{n \to \infty} \frac{1}{n} \log Z^{\ord}_{G_n} = f_\ord$ for $\abs{e^{\beta}-1}\ge e^{\beta_{0}} -1$.
\end{proof}

\subsection{Determining the critical point}
\label{subsecCritical}
We will define the critical point $\beta_c(q,\Delta)$ implicitly in terms of the
functions $f_\ord$ and $f_{\dis}$. It will be convenient to first
obtain the formula for $f_{\dis}$ given in Theorem~\ref{thmPhases}.

\begin{lemma}
  \label{lem:fdislem}
   For $\beta\in \mathbb C$ such that 
   $\abs{ e^{\beta}-1} \le e^{\beta_{1}}-1$
  \begin{equation}
    \label{eq:fdis}
    f_{\dis}(\beta,q) = \log q + \frac{\Delta}{2} \log\left(1+\frac{e^{\beta}-1}{q}\right)  \,.
  \end{equation}
\end{lemma}
\begin{proof}
  This proof uses the following generalization of
  Section~\ref{secDisPolyonTrees}. Given a finite subtree $T$ of
  $\mathbb T_\Delta$, define the disordered polymer model on $T$ just
  as we did for $\bT_\Delta^L$ and let $\cC_{\dis}(T)$ denote the
  collection of clusters of disordered polymers in $T$.  For a cluster
  $\Gamma$ let $\cG(\Gamma)$ denote the graph union of all polymers in
  $\Gamma$. As in Lemma~\ref{lemFreePartition}, the polymer model partition function on $T$ is a scaling of the random cluster model partition function.

  Note that the random cluster measure on a finite tree with free boundary
  conditions has a very simple description: it is independent edge
  percolation with the probability of retaining each edge being
  $\frac{e^{\beta} - 1 }{ e^{\beta}-1 + q }$~\cite[Chapter
  10]{grimmett2006random}.  This independence implies that any joint cumulant
  involving indicators of at least two edges vanishes, i.e., for
  all trees $T$ with at least two edges,
  \begin{equation}
    \label{eq:treecluster}
    \sum_{\substack{\Gamma\in \cC_{\dis}(T):\\  \cG(\Gamma)=T}} w^{\dis}(\Gamma)=0\,,
  \end{equation}
since the left-hand side is the joint cumulant of the edges of $T$ in
the random cluster model on $T$. 

 To conclude, note that we have
  \begin{equation*}
    \overline f_{\dis}=\sum_{\Gamma \in \cC_r^{\dis}}
    \frac{1}{u(\Gamma)} w^{\dis}(\Gamma) = \sum_T
    \frac{1}{|V(T)|}\sum_{\substack{\Gamma\in \cC_{\dis}(T):\\
        \cG(\Gamma)=T}}w(\Gamma)
  \end{equation*}
  where the first sum on the right hand side is over all finite
  subtrees of $\mathbb T_\Delta$ containing the root.  By
  Lemma~\ref{lemTreeConverge} part (1), these sums are absolutely
  convergent.  By \eqref{eq:treecluster}, only trees consisting of a
  single edge contribute to the sum over $T$, and there are $\Delta$
  of these.  Each contributes
  $ \log\left(1+\frac{e^{\beta}-1}{q}\right) $, and this gives the
  result.
\end{proof}

\begin{prop}
  \label{prop:unique}
  For all $\Delta \ge 5$ and $q=q(\Delta, \delta)$ large enough, there is a
  unique $\beta_{c}(q,\Delta) \in (\beta_0,\beta_1) $ such that
  $f_{\ord}(\beta)=f_{\dis}(\beta)$. Moreover, $f_{\ord} < f_{\dis}$ for
  $\beta \in [\beta_0,\beta_c)$ and $f_{\ord} > f_{\dis}$ for
  $\beta \in (\beta_c,\beta_1]$.
\end{prop}
\begin{proof}[Proof of Proposition~\ref{prop:unique}]
Our proof of this proposition follows the strategy
of~\cite{laanait1991interfaces}. 

We begin with a computation. Let $\beta\in [\beta_0, \beta_1]$ so that both the
ordered and disordered expansions converge.  Then
by Proposition~\ref{prop:limf} and Lemma~\ref{lem:fdislem},
\begin{align*}
  \label{eq:slope}
  \frac{d}{d\beta} (f_{\ord}-f_{\dis}) 
  &= 
  \frac{d}{d\beta} \lim_{n\to\infty}  \frac{1}{n}\log Z^{\ord}_{G_{n}}- \frac{\Delta}{2}\cdot\frac{e^\beta}{q+e^\beta-1}\\
 & =
  \lim_{n\to\infty}  \frac{1}{n}  \frac{d}{d\beta}  \log Z^{\ord}_{G_{n}}- \frac{\Delta}{2}\cdot\frac{e^\beta}{q+e^\beta-1}\, .
  \end{align*}
The interchange of the derivative and limit is valid since $f_{\ord}$ is a uniform limit of analytic functions by Proposition~\ref{prop:limf}. To bound the first term we note that 
\[
 \frac{1}{n}\cdot \frac{e^\beta-1}{e^\beta}  \frac{d}{d\beta}  \log Z^{\ord}_{G_{n}}
\]
is the expected number of edges in a random cluster configuration conditioned on $\Omega_{\ord}$ and is therefore at least $(1-\eta)n\Delta/2$. It follows that 
\[
  \frac{d}{d\beta} (f_{\ord}-f_{\dis}) \geq e^{\beta}\frac{\Delta}{2}\left[ \frac{1-\eta}{e^\beta-1}- \frac{1}{q+e^\beta-1}\right] >0,
\]
since $\eta\leq 1/100$ and $\beta\in [\beta_0, \beta_1]$.

Next, note that
\begin{equation}
  \label{eq:fdiff}
f_{\dis}-f_{\ord}= \frac{\Delta}{2}\log\left(\frac{q^{2/\Delta}}{e^\beta-1}+q^{2/\Delta-1} \right)- \overline f_{\ord} .
\end{equation}
By Lemma~\ref{lemTreeConverge} part (2) and Lemma~\ref{lemVertexTail} (with $\bT_\Delta$ in place of $G$) we have $|\overline f_{\ord}|\leq q^{\frac{1}{200\Delta}}$.
It follows that for $q$ sufficiently large, if $\beta=\beta_0$ then
$f_{\dis}>f_{\ord}$ and if $\beta=\beta_1$ then $f_{\ord}>f_{\dis}$.
Since $f_{\ord}-f_{\dis}$ is a continuous and strictly increasing
function of $\beta$ on $[\beta_0,\beta_1]$, we obtain that there is a unique $\beta_{c}\in(\beta_0,\beta_1)$ at which $f_{\ord}=f_{\dis}$.
\end{proof}

  \begin{cor}
    \label{cor:crit}
    For all $\Delta\geq 5$ and $q=q(\Delta)$ large enough,
    $\beta_{c}(q,\Delta)$ is given by
    \begin{equation}
      \label{eq:crit}
      \beta_{c}(q,\Delta) = (1+o_{q}(1))\frac{2\log q}{\Delta}.
    \end{equation}
  \end{cor}
  \begin{proof}
    By Proposition~\ref{prop:unique}, the claim follows by equating
    $f_{\ord}$ and $f_{\dis}$ and solving for $\beta$, see~\eqref{eq:fdiff}.
  \end{proof}

Proposition~\ref{prop:unique} implies there is a unique transition
point on locally tree-like sequences of finite graphs satisfying our
expansion hypotheses. The next proposition shows that the transition
is first order.
\begin{prop}
\label{propFreeEnergy}
For any sequence of locally tree-like graphs $G_n$ from
$\cG_{\Delta,\del}$, if $q=q(\Delta,\delta)$ is large enough then
\begin{enumerate}
\item For $\beta < \beta_c$,  $\limsup_{n \to \infty} \frac{1}{n} \log \mu_{n} ( \Omega \setminus \Omega_{\dis}) <0$. 
\item For $\beta > \beta_c$,  $\limsup_{n \to \infty} \frac{1}{n} \log \mu_{n} ( \Omega \setminus \Omega_{\ord}) <0$. 
\end{enumerate}
\end{prop}
\begin{proof}
  The lemma follows by combining Proposition~\ref{prop:limf} with the
  estimates of Lemma~\ref{lem:ErrQual}.
  If $\beta\in (\beta_0, \beta_c)$, then by Proposition~\ref{prop:limf} 
  and  Lemma~\ref{lem:ErrQual} part $(1)$, 
  \[
  \limsup_{n \to \infty} \frac{1}{n} \log \mu_{n} ( \Omega \setminus \Omega_{\dis}) \le \max \left \{ -1, f_{\ord}-f_{\dis} \right \}<0
  \] 
  where the last inequality follows from by Proposition~\ref{prop:unique}. Similarly if $\beta\in (\beta_c, \beta_1)$, then the quantity in part $(2)$ is at most $\max \{ -1, f_{\dis}-f_{\ord} \}<0$.
  
  If $\beta\leq \beta_0$, then by Lemma~\ref{lem:ErrQual}, 
  $\limsup_{n \to \infty} \frac{1}{n} \log \mu_{n} ( \Omega \setminus \Omega_{\dis}) \leq -1$.
  Similarly if $\beta\geq \beta_1$, then
  $\limsup_{n \to \infty} \frac{1}{n} \log \mu_{n} ( \Omega \setminus \Omega_{\ord}) \leq -1$.
\end{proof}

\subsection{Local convergence and proof of Theorem~\ref{thmPhases}}

Recall from Sections~\ref{sec:disord-polym-meas}
and~\ref{sec:ord-on-edge} that the disordered and ordered polymer
measures on a graph $G_{n}$ induce measures $\overline \nu_{\dis}^{n}$
and $\overline \nu_{\ord}^{n}$ on edges.
\begin{prop}
\label{propLocalConvergence}
Let $G_n$ be a sequence of locally tree-like graphs from $\cG_{\Delta,\del}$, $\Delta \ge 5$.  Then for $q$ large,
\begin{enumerate}
\item If $\beta \le \beta_1$, $\overline \nu_{\dis}^{n} \xrightarrow{loc} \mu^{\free}$.
\item If $\beta \ge \beta_0$, $\overline \nu_{\ord}^{n} \xrightarrow{loc} \mu^{\wire}$.
\end{enumerate}
\end{prop}
\begin{proof}
  We begin with the first statement. 
  To ease notation let us denote $\overline \nu_{\dis}^{n}$
  by $\overline \nu_{\dis}$. 
  Recall that for $T>0$ and $v\in G_n$, 
  $B_T(v)$ denotes the depth-$T$ neighborhood of $v$.
  Recall also that $\overline\nu_{\dis}^{B_{T}(v)}$ denotes the projection of
$\overline\nu_{\dis}$ to $\{0,1\}^{E(B_{T}(v))}$.

For $L>0$, let $\nu_{\dis, L}$ denote the Gibbs measure associated to
the disordered polymer model on $\bT_{\Delta}^L$ as defined in
Section~\ref{secDisPolyonTrees}.  By Lemma~\ref{lemFreePartition} the
induced measure $\overline \nu_{\dis, L}$ on
$\{0,1\}^{E(\bT_\Delta^L)}$ is the free random cluster measure on
$\bT_\Delta^L$.

We let $r$ denote the root of the tree $\bT_{\Delta}^L$.
We will show that for $T>0$ and $\eps>0$, for all $L$ sufficiently large
 the distribution $\overline\nu_{\dis}^{B_{T}(v)}$ of a randomly
 chosen $v \in G_n$ is within distance $2\eps$ of $\overline
 \nu_{\dis, L}^{B_{T}(r)}$ in total variation distance.
 This suffices to prove part (1) since $\mu^{\free}$ is the weak limit of 
$\overline \nu_{\dis, L}$ as $L\to \infty$.

We will apply Lemma~\ref{lemDepthDistribution}.  Given $\eps>0$, let
$m= m(\Delta, T,\eps)$ large enough as required by the lemma.  Since
$G_n$ is locally tree-like, with high probability over the choice of
$v$, the depth-$m$ neighborhood of $v$ will be a tree, so we can
condition on this.  Lemma~\ref{lemDepthDistribution} tells us that up
to total variation distance $\eps$, $\overline\nu_{\dis}^{B_{T}(v)}$
is determined by clusters contained in $B_m(v)$.
  
  By Lemma~\ref{lemTreeConverge}, the cluster expansion of the
  disordered polymer model on $\bT_{\Delta}^L$ converges for all
  $L$. For $L\geq m$, we may apply the proof of
  Lemma~\ref{lemDepthDistribution} to show that up to total variation
  distance $\eps$, $\overline \nu_{\dis, L}^{B_{T}(r)}$ is determined
  by clusters contained in $B_m(r)$.
  
  Since $B_m(v)$ and $B_m(r)$ are identical, we have
  \begin{equation*}
  \| \overline\nu_{\dis}^{B_{T}(v)} - \overline \nu_{\dis, L}^{B_{T}(r)}\|_{TV}\leq 2\eps
  \end{equation*}
 as required.  

  The proof of the second claim is identical, using
  Lemma~\ref{lemWirePartition} in place of
  Lemma~\ref{lemFreePartition}.
\end{proof}

\begin{proof}[Proof of Theorem~\ref{thmPhases}]
  Claim (1) follows from Proposition~\ref{prop:limf}. The limit
  $\lim_{n \to \infty} \frac{1}{n} \log Z_{G_n}$ is analytic for
  $\beta \in (0 , \infty) \setminus \{ \beta_c \}$ since $f_{\dis}$ is
  analytic on $(0, \beta_1]$ and $f_{\ord}$, as a uniform limit of
  analytic functions, is analytic on $[\beta_0, \infty)$. The formula
  for $f_{\dis}$ when $\beta<\beta_{c}$ follows from
  Lemma~\ref{lem:fdislem}. Claims (2) and (3) follow immediately from
  Proposition~\ref{propFreeEnergy}. Claim (8) follows from
  Lemma~\ref{lem:ErrQual}.

  To emphasize the dependence on $n$, write $\mu^{n}_\dis$ and
  $\mu_\ord^{n}$ denote the distributions of $\mu_n$ conditioned on
  $\Omega_{\dis}$ and  $\Omega_{\ord}$, respectively.  To conclude, we will
  prove the following strengthening of Claims (4), (5), (6), and (7):
\begin{enumerate}[label=(\roman*)]
\item For $\beta \leq \beta_1$, $ \mu_{\dis}^{n} \xrightarrow{loc} \mu^{\free}$ as $n\to\infty$. 
\item For $\beta \geq \beta_0$, $ \mu_{\ord}^{n} \xrightarrow{loc} \mu^{\wire}$ as $n\to\infty$. 
\item For $\beta \le \beta_1$, $\mu_{\dis}^{n}$ exhibits exponential decay of correlations and $|\mathbf A|$ obeys a central limit theorem with respect to $\mu_{\dis}^{n}$.   
\item For $\beta \ge \beta_0$, $\mu_{\ord}^{n}$ exhibits exponential decay of correlations and $|\mathbf A|$ obeys a central limit theorem with respect to $\mu_{\ord}^{n}$.  
\end{enumerate}
Given Lemmas~\ref{lem:DisQual} and~\ref{lem:OrdQual} it is enough to
prove (i)--(iv) for $\overline \nu^{n}_{\dis}$ and $\overline \nu^{n}_{\ord}$
in place of $\mu^{n}_{\dis}$ and $\mu^{n}_{\ord}$.  Claims (i) and (ii) then
follow from Proposition~\ref{propLocalConvergence}. Claims (iii) and
(iv) follow from Lemma~\ref{lemPolyExpDecay} combined with the
observation that since the diameter of an expander graph is
$O(\log n)$, the total variation distance error $e^{-n}$ from
Lemmas~\ref{lem:DisQual} and~\ref{lem:OrdQual} can be absorbed in the
constant in the exponential decay bound.
\end{proof}

  \section{Slow mixing of Markov chains}
  \label{secSlow}
In this section we prove Theorem~\ref{thmSlowMix}.
We will give the proof for \emph{Chayes-Machta (CM)
 dynamics}~\cite{chayes1998graphical} and then indicate how to adapt the proof
for the (much simpler) case of random cluster and Potts Glauber dynamics.

  We begin
by recalling the definition of the \emph{Chayes-Machta (CM)
  dynamics}~\cite{chayes1998graphical}, a generalization of
Swendsen-Wang dynamics for the Potts model to the setting of the
random cluster model.  Given a random cluster configuration
$A\in \Omega=\{0,1\}^E$, one step of the CM dynamics is defined as
follows:
\begin{enumerate}
\item declare each component of $A$ to be `active' independently with
  probability $1/q$, and declare all vertices in active components to be active;
\item delete all edges in $A$ that connect two active vertices;
\item add each edge in $E$ that connects two active vertices independently with probability $p=1-e^{-\beta}$.
\end{enumerate}
We use $P_{\mathrm{CM}}(\cdot,\cdot)$ to denote the transition matrix of the CM dynamics, and $\mu_A^t$ the $t$-step distribution of the chain started at configuration $A$.  The \textit{mixing time} of the CM dynamics is:
\[ \tau_{\mathrm{mix}} = \inf \{ t: \max_{A \subset E} \| \mu - \mu_A^t\|_{TV} \le 1/4 \,. \]

Our general strategy follows one previously used at $\beta=\beta_c$, e.g.,~\cite{borgs2012tight,galanis2016ferromagnetic}. Our ability to extend slow mixing to an interval around $\beta_c$ stems from our ability to control the contribution of subdominant phases off criticality (Lemma~\ref{lembetainterval} below).

We begin with a lemma that says CM dynamics are
unlikely to transition from an ordered configuration to a disordered configuration.
\begin{lemma}\label{lemconductance}
For $q=q(\Delta)$ sufficiently large and $\beta\in (\beta_0, \beta_1)$, $P_{\mathrm{CM}}(A, \Omega_\dis)<e^{-n\Delta/40}$ for all $A\in\Omega_\ord$.
 \end{lemma}
\begin{proof}
  Let $U\subseteq V$ denote the set of vertices declared active at
  Step 1 in the definition of CM dynamics and let $A'$ denote the
  random edge configuration resulting from Steps 1,2 and 3.  Let
  $m=|E\cap\binom{U}{2}|$, that is, the number of edges of $G$ joining
  two active vertices.  Note that the number of edges removed from the
  configuration in Step 2 is at most $m$ and so if $m<|E|/2$, then
  $|A'|>|A|-|E|/2\geq (1/2-\eta)|E|$.  Therefore
  $A'\notin \Omega_\dis$ and so we may assume that $m\geq|E|/2$.
 
  Letting $X$ denote the number edges added at Step 3 we have
  $|A'|>|A|-m+X$. If $X\geq pm/2$, it follows that
\[
  |A'|\geq (1-\eta) |E| -(1-p/2)m \geq (p/4
  -\eta) |E| >\eta|E|
\]
Since $p=1-e^{-\beta}=1-o_{q}(1)$ for $\beta\in (\beta_0, \beta_1)$,
we have $p>8\eta$ for $q$ large. As a result, $A'\notin \Omega_\dis$. 
The result follows by noting that $\P(X<pm/2)\leq e^{-pm/8}$ by Chernoff's bound.
\end{proof}

The next lemma says that near $\beta_{c}$, it is exponentially more
likely to see a disordered or ordered configuration than a
configuration in $\Omega_{\err}$.
\begin{lemma}\label{lembetainterval}
If $q=q(\Delta,\delta)$ is sufficiently large,
  $|\beta-\beta_c|\leq \frac{1}{20\Delta}$, and $n$ is sufficiently large, then
\[
\mu_{n}(\Omega_\dis)\geq e^{-n/20} \quad\text{and}\quad \mu_{n}(\Omega_\ord)\geq e^{-n/20}\, .
\]
\end{lemma}
\begin{proof}
  If $\beta\in (\beta_0, \beta_1)$, then by Proposition~\ref{prop:limf} 
  and  Lemma~\ref{lem:ErrQual} part $(1)$, 
  \[
  \lim_{n \to \infty} \frac{1}{n} \log \mu_{n} (\Omega_{\dis}) \ge \min \left \{ 0, f_{\dis}-f_{\ord} \right \}\, .
  \] 
  By the argument of Proposition~\ref{prop:unique} we have 
  \[
  \left| \frac{d}{d\beta} (f_{\ord}-f_{\dis}) \right| < \frac{e^\beta}{e^{\beta}-1}\cdot \frac{\Delta}{2}< \Delta\, .
  \]
  Since $f_{\dis}-f_{\ord}=0$ at $\beta=\beta_c$, it follows that
  $f_{\dis}-f_{\ord}> - 1/20$ for
  $\beta\in(\beta_0, \beta_c + \tfrac{1}{20 \Delta})$. The bound on
  $\mu_{n}(\Omega_\dis)$ follows. The same argument shows that
  $\lim_{n \to \infty} \frac{1}{n} \log \mu_{n} (\Omega_{\ord})>-1/20$
  for $\beta\in(\beta_c- \tfrac{1}{20 \Delta}, \beta_1)$.
\end{proof}

\begin{proof}[Proof of Theorem~\ref{thmSlowMix} for CM dynamics]
  We will establish slow mixing of CM dynamics by bounding the
  \emph{conductance} of CM dynamics defined as
\[
\Phi_{CM}=\min_{\emptyset \subset S \subset \Omega}\Phi_{CM}(S) \quad\text{where}\quad \Phi_{CM}(S)= \frac{\sum_{A\in S}\mu(A)P_{\mathrm{CM}}(A, S^c)}{\mu(S)\mu(S^c)} \,.
\]
Note
  that $P_{\mathrm{CM}}$ and $\mu$ depend on the given graph $G=(V,E)$, and in particular, on
  $n=|V(G)|$. We leave this implicit.
By a standard argument (see~\cite{montenegro2006mathematical}), it suffices to show that $\Phi_{CM}\leq e^{-\Omega(n)}$
for $\beta\in(\beta_{m},\beta_{M})$. This is straightforward from the
lemmas above: 
\begin{align*}
\Phi_{CM}
&\leq\Phi_{CM}(\Omega_\dis)=\frac{\sum_{A\in\Omega_\dis}\mu(A)P_{\mathrm{CM}}(A, \Omega_\dis^c)}{\mu(\Omega_\dis)\mu(\Omega_\dis^c)}\\
&\leq e^{n/10}\left(\sum_{A\in\Omega_\dis}\mu(A)P_{\mathrm{CM}}(A, \Omega_\ord)+\sum_{A\in\Omega_\dis}\mu(A)P_{\mathrm{CM}}(A, \Omega_\err) \right)\\
&\leq e^{n/10}\left(\mu(\Omega_\dis)e^{-n\Delta/40}+\sum_{A\in\Omega_\err}\mu(A)P_{\mathrm{CM}}(A, \Omega_\dis) \right)\\
&\leq e^{n/10}\left(e^{-n\Delta/40}+e^{-n}\right)\\
&\leq2e^{-n/40}\, .
\end{align*}
For the second inequality we used Lemma~\ref{lembetainterval}.
For the third inequality we applied Lemma~\ref{lemconductance} and reversibility, 
and for the fourth inequality we used Lemma~\ref{lem:ErrQual}.
\end{proof}
We conclude this section by noting that the above proof adapts easily
to the to the cases of random cluster and Potts model Glauber
dynamics.  First we recall their definitions. Given a random cluster
configuration $A\in \Omega=\{0,1\}^E$, one step of the \emph{random
  cluster Glauber dynamics} transitions to a new configuration $A'$ as
follows:
\begin{enumerate}
\item select an edge $e\in E$ uniformly at random;
\item set $A'=A\cup\{e\}$ with probability $\frac{\mu_G(A\cup\{e\})}{\mu_G(A\cup\{e\})+\mu_G(A\backslash\{e\})}$
\item otherwise set $A'=A\backslash \{e\}.$
\end{enumerate}

Given a Potts configuration $\sigma\in[q]^V$, one step of the \emph{Potts model Glauber dynamics} transitions to a new configuration $\sigma'$ as follows:
\begin{enumerate}
\item select a vertex $v\in V$ uniformly at random;
\item set $\sigma'(v)=k$ with probability $\mu_G^{\Potts}(\tau(v)=k \mid \tau(w)=\sigma(w)\, \,   \forall w\neq v)$ and set  $\sigma'(u)=\sigma(u)$ for all $u\neq v$.
\end{enumerate}
We note that by considering the monochromatic edges in each Potts model configuration,
the above dynamics naturally induces dynamics on the random cluster model.
The proof of Theorem~\ref{thmSlowMix} for the Glauber dynamics for the
random cluster model and Potts model are similar to the proof for CM dynamics but simpler, as the associated dynamics cannot
transition directly from $\Omega_{\dis}$ to $\Omega_{\ord}$. We omit the details.

\section{Application to random $\Delta$-regular graphs}
\label{sec:RRGexp}

In this section we prove that for $\Delta \ge 5$, there is
some $\del>0$ so that the random $\Delta$-regular graph belongs to
$\cG_{\Delta,\del}$ with high probability. We use the following result
on the expansion profile of the random regular graph which is a
combination of \cite[Theorem 1]{bollobas1988isoperimetric} and
\cite[Theorem 4.16]{hoory2006expander}.

\begin{theorem}\label{thmExpander}
Let $\Delta\geq3$ and let $G$ be a 
$\Delta$-regular graph on $n$ vertices chosen 
uniformly at random. Let $0<x<1$ be such that
\begin{equation}
  \label{expalpha}
2^{4/\Delta}<(1-x)^{1-x}(1+x)^{1+x}\, ,
\end{equation}
then with high probability $\phi_G(1/2)\geq (1-x)/2$.
Moreover, for every $\eps>0$, there exists $\delta>0$ 
such that with high probability $\phi_G(\delta)\geq(\Delta-2-\eps)/\Delta$.
\end{theorem}

\begin{prop}
  \label{prop:RRG-exp}
  For every $\Delta \ge 5$, there exists $\del >0$ so that a uniformly
  chosen $\Delta$-regular graph on $n$ vertices is in $\cc G_{\Delta,\delta}$ with
 probability $1-o(1)$ as $n \to \infty$.  Moreover, there is a polynomial-time algorithm
  that accepts/rejects graphs that (i) only accepts $G$
  if 
  $G \in \cc G_{\Delta,\delta}$ and (ii) it accepts with probability
  $1-o(1)$ for a randomly chosen $\Delta$-regular graph.
\end{prop}
\begin{proof}
Let $G$ be a uniformly
 chosen $\Delta$-regular graph on $n$ vertices. 
 By substituting $x=1/10$ into~\eqref{expalpha},
 Theorem~\ref{thmExpander} shows that $\phi_{G}(1/2)\geq 1/10$ 
 with high probability. Moreover, taking $\eps=2/9$, the second half of Theorem~\ref{thmExpander} 
shows that there exists $\delta>0$ such that  
$\phi_{G}(\delta)\geq 5/9$ with high probability.
 This proves the first claim.  
 We remark that one can extract explicit
  sufficient conditions on $\delta$ from the proof of \cite[Theorem
  4.16]{hoory2006expander}.

Note that it also holds that 
for some $\tilde\eps>0$, $\phi_{G}(1/2)\geq 1/10 + \tilde \eps$ and 
$\phi_{G}(\delta)\geq 5/9 +\tilde\eps$
with high probability.   Then using the approximation algorithm from~\cite{guruswami2013rounding}, we can approximate $\phi_{G}(1/2)$ and $\phi_{G}(\delta)$ and with high probability get a certificate that $\phi_{G}(1/2)\geq 1/10$ and $\phi_{G}(\delta)\geq 5/9$. 
\end{proof}

\section{Finite Size Scaling}
\label{sec:finite-size-scaling}

In this section $G_{n}$ is always a random $\Delta$-regular graph on
$n$ vertices.  Our objective is to determine the limiting distribution
of $\log Z^{\ord}_{G_n} - n f_{\ord} $ and
$\log Z^{\dis}_{G_n} - n f_{\dis} $ as $n \to \infty$.  This
will prove Theorems~\ref{thmCritical} and~\ref{thm:Q}. To ease
notation, we let $Z^{\dis}$ and  $Z^{\ord}$ denote
$Z^{\dis}_{G_n}$ and $Z^{\ord}_{G_n}$, respectively.

To state the key proposition we need to introduce a class of graphs that will
capture the way in which a $\Delta$-regular graph locally deviates
from being a tree.  Assume $\Delta\geq 3$. Let $\bT_{\Delta-2,\Delta}$
denote the rooted infinite tree whose root has $\Delta-2$ children and
for which every other vertex is degree $\Delta$. For $k\geq 3$ the
\emph{$\Delta$-regular tree rooted at $C_{k}$} is the graph
$\bT_\Delta^{C_{k}}$ obtained by attaching to each vertex of a
$k$-cycle $C_{k}$ a copy of $\bT_{\Delta-2,\Delta}$, and rooting the
resulting graph at a distinguished vertex $r$ in $C_{k}$.

We define disordered and ordered polymers on $\bT_\Delta^{C_k}$
exactly as we did for $\bT_\Delta$ in
Section~\ref{sec:limit-free-energ}. For $q=q(\Delta)$ large enough the
cluster expansions for these polymer models converge provided
$\beta\leq \beta_{1}$ and $\beta\geq \beta_{0}$, respectively. This
can be established by repeating the proof of Lemma~\ref{lemTreeConverge}. For
$\ast\in \{\dis,\ord\}$ let $\cC_{\ast}^r(\bT_\Delta^{C_{k}})$ denote
the set of $\ast$-clusters that contain the root $r$ of
$\bT_{\Delta}^{C_{k}}$.  To help distinguish notation, in this section
we write $\cC_{\ast}^{r}(\bT_\Delta)$ for the sets of $\ast$-clusters
on the rooted $\Delta$-regular tree that contain the root.  We then
let
\begin{align*}
\alpha_k^{\dis} &\bydef \sum_{ \Gamma \in \cC_{\dis}^{r}(\bT_\Delta^{C_{k}})} w^{\dis}(\Gamma)-\sum_{\Gamma \in \cC_{\dis}^r(\bT_\Delta)} w^{\dis}(\Gamma)  \\
\alpha_k^{\ord} &\bydef \sum_{ \Gamma \in \cC_{\ord}^r(\bT_\Delta^{C_{k}})} w^{\ord}(\Gamma)-\sum_{\Gamma \in \cC_{\ord}^r(\bT_\Delta)} w^{\ord}(\Gamma) \,.
\end{align*}
Note that $\alpha_k^{\dis}$ is well-defined as the difference of two
absolutely convergent power series when $\beta \le \beta_1$, and
similarly for $\alpha_k^{\ord}$ when $\beta \ge \beta_0$.

\begin{prop}
\label{propQDistribution}
Let $(Y_1, Y_2, \dots)$ be a sequence of independent Poisson random variables where $Y_k$ has mean $(\Delta-1)^{k}/(2k)$.  
\begin{enumerate}
\item For $\beta \le \beta_1$, $W_n^{\dis} \bydef \log Z^{\dis} - n f
  _{\dis}$ converges in distribution to $W^{\dis}$ given by the almost
  surely absolutely convergent series
\begin{align*}
W^{\dis} \bydef \sum_{ k \ge 3}  \alpha_k^{\dis} Y_k \,.
\end{align*}
\item For $\beta \ge \beta_0$, $W_n^{\ord} \bydef \log Z^{\ord} - n f
  _{\ord}$ converges in distribution to $W^{\ord}$ given by the almost
  surely absolutely convergent series
\begin{align*}
 W^{\ord} \bydef \log q +  \sum_{ k \ge 3}  \alpha_k^{\ord} Y_k \,.
\end{align*}
\item For $\beta =\beta_c$, $W_n \bydef \log Z^{\ord}  - \log Z^{\dis}
  $ converges in distribution to $ W$ given by the almost surely
  absolutely convergent series
\begin{align*}
W \bydef   \log q +  \sum_{ k \ge 3}  (\alpha_k^{\ord} - \alpha_k^{\dis}) Y_k \,.
\end{align*}
Moreover, letting $Q\bydef e^W$, we have $Q/q \to 1 $ in probability as $q \to \infty$. 
\end{enumerate}
\end{prop}

To prove Proposition~\ref{propQDistribution} we will need several results about the distribution of short cycles in random $\Delta$-regular graphs.
\begin{lemma}
\label{lemShortCycles}
For $k \ge 3$, let $X_k$ denote the number of cycles of length $k$ in the random $\Delta$-regular  graph on $n$ vertices. Then
\begin{enumerate}
\item  For $3 \le k \le  \frac{\log n}{5 \log \Delta}$, $\E X_k = (1+O(k^2/n)) \frac{ (\Delta-1)^k}{2k}$~\cite{mckay2004short}. 
\item For any fixed $T$, the joint distribution of $X_3, \dots, X_T$ converges to that of independent Poisson random variables of means $ \frac{ (\Delta-1)^k}{2k}$, $k=3, \dots, T$~\cite{wormald1981asymptotic,bollobas1980probabilistic}.  
\item For every fixed $T>0$, with high probability over the choice of $G$, the depth-$t$ neighborhood of every vertex contains at most one cycle of length at most $T$ for $t =  \frac{\log n}{5 \log \Delta}$~\cite[Lemma 2.1]{lubetzky2010cutoff}. 
\end{enumerate}
\end{lemma}

With this we can prove Proposition~\ref{propQDistribution}.
\begin{proof}
  Throughout this proof we assume
  the high probability event that $G_{n}\in
\cG_{\Delta,\delta}$, with $\delta$ chosen as in
Section~\ref{sec:RRGexp} occurs. We begin by establishing the claimed limits, deferring the claims
about absolute convergence almost surely to the end. Towards claim
(1), let $X_{k}$ denote the number of $k$-cycles in $G_{n}$, and define
\begin{equation*}
  \tilde W^{\dis}_{n}(T) \bydef \sum_{ k = 3}^T  \alpha_k^{\dis} X_k\, .
\end{equation*}
We will show that for any $\eps >0$, there is $T$ large enough so that  for all $t \ge T$ we have
\begin{align*}
\limsup_{n \to \infty} \P \left[ \left| \tilde W^{\dis}_{n}(t) 
-  W_{n}^{\dis} 
  \right| \ge \eps  \right ]\le \eps \,.
\end{align*} 
By Lemma~\ref{lemShortCycles}, part (2), the joint distribution of
$X_3, \dots , X_t$ converges to that of $Y_3, \dots, Y_t$, and so this
will prove that $W_{n}^{\dis}=\log Z^{\dis} - n f _{\dis}$ converges
to $W^{\dis}=\sum_{k \ge 3} \alpha_k^{\dis} Y_k$ in distribution.

Fix $\eps>0$. We begin with the formula from the proof of
Proposition~\ref{prop:limf} for $W_n^{\dis}$. Writing $G$ in place of
$G_{n}$,
\begin{equation*}
W_n^{\dis} =  \sum_{v \in V}    \left( \sum_{\Gamma \in \cC_{\dis}^{v} (G)} \frac{1}{u(\Gamma)}w^{\dis}(\Gamma)   -   \sum_{\Gamma \in \cC^{r}_{\dis}(\bT_\Delta)} \frac{1}{u(\Gamma)} w^{\dis}(\Gamma) \right)  .
\end{equation*}
Let $m = 200 \Delta \log (4n/\eps)/ \log q$. We can apply
Lemma~\ref{lemVertexTail} to obtain
\begin{equation}
  \label{eq:Qbrack}
\left| W_n^{\dis} -  \sum_{v \in V}    \left( \mathop{\sum_{\Gamma\in\cC_{\dis}^{v}(G)}}_{\|\Gamma\|\le m} \frac{1}{u(\Gamma)}w^{\dis}(\Gamma)   -   \mathop{\sum_{ \Gamma \in \cC^{r}_{\dis}(\bT_\Delta)}}_{\|\Gamma \| \le m } \frac{1}{u(\Gamma)} w^{\dis}(\Gamma)\right )   \right |  \le \eps /2  \,.
\end{equation}
 Note that the terms inside the
parentheses cancel exactly unless there is a cycle, necessarily of
length at most $m$, in the $m$-neighborhood of $v$. To measure the
error in $\tilde W^{\dis}_{n}(T) - W^{\dis}_{n}$ due to these cycles
we will reformulate~\eqref{eq:Qbrack} in a way that takes cancellations into account. 

A cluster $\Gamma$ appears in only one of the two sums in~\eqref{eq:Qbrack} for only two possible reasons: because the cluster contains a cycle in $G$, or because the cycle prevents a cluster on the tree from occurring in $G$.  For a cycle of length $k$ these possibilities only occur for clusters of size at least $k$ because smaller clusters in $G$ match the tree clusters exactly.  To account for the fact that a single cycle will appear in the neighborhood of many vertices in the sum above, we instead sum over the cycles in $G$ and remove the factor $1/u(\Gamma)$.  Formally, for each cycle $C$ let $v(C)$ be a distinguished vertex on the cycle, and let $\mathrm{Cyc}_m(G)$ be the set of all cycles of length at most $m$ of $G$.  Then the sum over $v$ in $V$ in~\eqref{eq:Qbrack} can be rewritten as
\begin{equation*}
  W^{\dis}_{n}(m)\bydef \sum_{C \in \mathrm{Cyc}_m(G)} \left( \mathop{\sum_{\Gamma\in\cC_{\dis}^{v(C)}(G)}}_{|C| \le \|\Gamma \|\le m} 
w^{\dis}(\Gamma) - \mathop{\sum_{\Gamma \in \cC_{\dis}^{r}(\bT_\Delta)}}_{|C| \le \|\Gamma \| \le m } w^{\dis}(\Gamma)\right ) \, .
\end{equation*}
If $G$ satisfies conclusion (3) of Lemma~\ref{lemShortCycles} then there is at most one cycle in the depth-$m$ neighborhood of each vertex, and hence
\begin{equation*}
  W^{\dis}_{n}(m) = \sum_{C \in \mathrm{Cyc}_m(G)} \left( \mathop{\sum_{\Gamma\in\cC_{\dis}^{r}(\bT_\Delta^{C})}}_{|C| \le \|\Gamma \|\le m} 
w^{\dis}(\Gamma) - \mathop{\sum_{\Gamma \in \cC_{\dis}^{r}(\bT_\Delta)}}_{|C| \le \|\Gamma \| \le m } w^{\dis}(\Gamma)\right ), 
\end{equation*}
where $\bT_\Delta^{C}$ is the $\Delta$-regular tree rooted at the
cycle $C$. We have used here that any cluster containing a polymer
that is not contained in the $m$-neighborhood of $r$ has size larger
than $m$, so there is no need to truncate $\bT_\Delta^{C}$ to a finite
depth. Moreover, \eqref{eq:Qbrack} can be rewritten as
\begin{equation} 
  \label{eq:Qbrack2}
\left| W_n^{\dis} - W_{n}^{\dis}(m)\right| \leq \eps/2. 
\end{equation}
By Lemma~\ref{lemVertexTail} we have that 
\begin{equation*}
  \left|\mathop{\sum_{\Gamma\in\cC_{\dis}^{r}(\bT_\Delta^{C})}}_{|C| \le \|\Gamma \|\le m} 
w^{\dis}(\Gamma) - \mathop{\sum_{\Gamma \in \cC_{\dis}^{r}(\bT_\Delta)}}_{|C| \le \|\Gamma \| \le m } w^{\dis}(\Gamma) \right | \le 2 q^{-|C|/200 \Delta}  
\end{equation*}
Using Lemma~\ref{lemShortCycles} part (1) this means the expected contribution to the error from cycles of length at least $T$ is at most 
\begin{equation*}
\E \left| \tilde W^{\dis}_{n}(T) - W^{\dis}_{n}(m)\right| \leq \sum_{t \ge T}2 \frac{(\Delta-1)^t}{2t} q^{-t/200 \Delta}  \,.
\end{equation*}
Then if $q \ge \Delta^{400 \Delta}$, the expected contribution is at most $\Delta^{-T}$.  If we take $T = |\log_\Delta (\eps^2/4)|$ then by Markov's inequality
\begin{equation*}
  \P \cb{\left| \tilde W^{\dis}_{n}(t) - W^{\dis}_{n}(m)\right| \geq \frac{\eps}{2}} \leq \frac{\eps}{2}.
\end{equation*}
for all $t\geq T$. Combining  this with~\eqref{eq:Qbrack2} we obtain
\begin{equation*}
  \P \cb{\left| \tilde W^{\dis}_{n}(t) - W^{\dis}_{n}\right| \geq \eps} \leq \eps,
\end{equation*}
for all $t\geq T$ as desired.

Part (2) of Proposition~\ref{propQDistribution} for $W_n^{\ord}$ can
be proven in the same way, and part (3) follows by combining the first
two parts since the cycle counts are coupled identically.  Next we
show that $Q/q\to 1$ in probability as $q\to\infty$.  It suffices to
prove that
\begin{equation*}
  \E \cb{ \left |\sum_{k\geq 3}(\alpha_k^{\ord}-\alpha_k^{\dis})Y_{k}
    \right| }  = o_q \left( 1 \right) 
\end{equation*}
as $q\to\infty$.  We can bound this by
\begin{align*}
 \E \cb{ \left |\sum_{k\geq 3}(\alpha_k^{\ord}-\alpha_k^{\dis})Y_{k} \right| }  &\le \sum_{k \ge 3} \E [Y_k] \left |\alpha_k^{\ord}-\alpha_k^{\dis} \right | \\
 &\le \sum_{k \ge 3}  \frac{(\Delta -1)^k}{2k} \cdot 4q^{-\frac{k}{200\Delta}} \\
 &\le \sum_{k \ge 3}  \exp \left [k\ob{ \log (\Delta-1) - \frac{\log q }{200 \Delta}  }  \right ] \\
 &= o_q \left( 1 \right)
\end{align*}
for $q=q(\Delta)$ sufficiently large, i.e., $q \ge \Delta^{400 \Delta}$. 

To conclude, observe that this last calculation (and exactly analogous
computations for parts (1) and (2)) verifies the conditions of Kolmogorov's
two-series theorem, implying the claimed almost sure absolute
convergence.
\end{proof}

\begin{proof}[Proof of Theorem~\ref{thmCritical}]
  Claim (1) follows by combining claims (1) and (2) of
  Proposition~\ref{propQDistribution}, and claim (3) is part of
  Proposition~\ref{propQDistribution} claim (3). 

Claim (2) follows from the stronger statements in the proof of Theorem~\ref{thmPhases}  that for $\beta \leq \beta_1$, $ \mu_{\dis}^{n} \xrightarrow{loc} \mu^{\free}$, and for  $\beta \geq \beta_0$, $ \mu_{\ord}^{n} \xrightarrow{loc} \mu^{\wire}$.   That proof also implies that the conditional measures both exhibit exponential decay of correlations (and central limit theorems) at $\beta_c$.
\end{proof}

\begin{proof}[Proof of Theorem~\ref{thm:Q}] The first two parts of
  this proposition are special cases of
  Proposition~\ref{propQDistribution}. The third part follows
  from the first two and Lemma~\ref{lem:ErrQual} as in the proof of
  Theorem~\ref{thmCritical} above.
\end{proof}

\section{Algorithms}
\label{sec:algor-cons}

The polymer models and estimates in Section~\ref{sec:polym-model-repr}
yield efficient approximate counting and sampling algorithms by
adapting the polymer model algorithms
from~\cite{HelmuthAlgorithmic2,JKP2,PottsAll2019stoc} to our current
setup.  In particular, if we assume $\eps > e^{-n/2}/2$, then by
Lemmas~\ref{lem:DisQual} and~\ref{lem:OrdQual} and
Corollary~\ref{cor:polyapprox} it suffices to find an FPTAS for
$\Xi^{\dis}$ when $\beta \le \beta_1$ and for $\Xi^{\ord}$ when
$\beta \ge \beta_0$, as well as polynomial-time sampling algorithms
for $\overline \nu_{\dis}$ and $\overline \nu_{\ord}$.

\subsection{Approximate counting}
\label{sec:approx}

The approximate counting algorithms
from~\cite{HelmuthAlgorithmic2,JKP2,PottsAll2019stoc} based on
truncating the cluster expansion have two main requirements: 1)
condition~\eqref{eq:approx}, or a similar statement giving an
exponentially small error bound, holds, and 2) one can list all polymers of
size at most $m$ and compute their weight functions in time
$\exp ( O(m +\log n))$.

\begin{lemma}
  \label{lem:ord-alg}
  There is an algorithm that lists all ordered polymers of size at
  most $m$ with running time $\exp(O(m + \log n))$.
\end{lemma}
\begin{proof}
  We can enumerate all connected edge sets of size at most $m$ in time $n
  (e\Delta)^{m}$, and hence can create a list $L$ of all labelled
  connected edge sets of size at most $m$ in time $m \exp(O(m))$. 

  By Lemma~\ref{lem:ord-gen} it takes time $\mathrm{poly}(K)$ to
  create the ordered polymer corresponding to set of $K$ unoccupied
  edges. For each element $(\gamma,\ell)$ of $L$ we apply the ordered polymer
  construction to the subset of edges of $\gamma$ labelled `unoccupied'. If
  this construction returns $(\gamma,\ell)$ we retain the element,
  otherwise we discard it. In time $\mathrm{poly}(m)
  \exp(O(m))=\exp(O(m))$ we obtain a complete list of all ordered polymers.
\end{proof}

\begin{proof}[Proof of Theorem~\ref{thmAlgorithm}, FPTAS]
  Suppose $\eps>e^{-n/2}$. 
  Note that a polynomial-time algorithm to
  compute $T_{m}^{\ord}$ (as defined in~\eqref{eq:Tmdef}) with $m=O(\log n/\eps)$ yields a
  polynomial-time algorithm for an $\epsilon$-approximation to
  $\Xi^{\ord}$ by~\eqref{eq:approx} whenever $\beta> \beta_{0}$, as in this
  case the cluster expansion converges. The same statement holds true
  for $\beta<\beta_{1}$ for approximation $\Xi^{\dis}$ by
  $T_{m}^{\dis}$. In turn, Corollary~\ref{cor:polyapprox} implies that this
  gives an FPTAS for $Z$. The existence of a polynomial-time algorithm
  to compute $T_{m}^{\ord}$ and $T_{m}^{\dis}$ can be seen as follows.

  By~\cite[Lemma~2.2]{PottsAll2019stoc}, an algorithm for computing
  $T_{m}^{\ord}$ (resp.\ $T_{m}^{\dis}$) exists provided there are
  polynomial-time algorithms to
  \begin{enumerate}
  \item list all polymers of given size $m\leq O(\log n)$,
  \item compute the weights of all polymers of size $m\leq O(\log n)$.
  \end{enumerate}

  Listing and computing the weights of disordered polymers of size $m$
  in time exponential in $m$ is elementary: all connected subgraphs on
  $k$ edges can be listed in time $n (e\Delta)^{k}$, and the weights
  are given by an explicit formula. 

  Listing ordered polymers of a given size $m$ in time exponential in
  $m$ can be done by Lemma~\ref{lem:ord-alg}. To compute the weights
  in time polynomial in $n$ requires computing the number $c'(\gamma)$
  of connected components induced by a polymer. Since it takes time
  $\abs{C(x)}$ to determine the connected component $C(x)$ of a vertex
  $x$, this can be done in time $n$ for each polymer. 

  This completes the proof when $\eps>e^{-n/2}$.
  When $\eps\leq e^{-n/2}$, one can obtain an FPTAS 
  by brute-force enumeration, as the total number
  of configurations is $2^{\Delta n/2}$.
\end{proof}

\subsection{Approximate sampling}
\label{sec:sample}

Since an efficient sampling algorithm for the Potts model when $q$ is
a positive integer follows from an efficient algorithm for the random
cluster model by the Edwards--Sokal coupling, we describe our
efficient sampling algorithm only for the random cluster model.

\begin{proof}[Proof of Theorem~\ref{thmAlgorithm}, Sampling] We
  consider only $\eps>e^{-n/2}$, as smaller $\epsilon$ can be handled
  by brute force. 

  When the disordered (respectively, ordered) cluster expansion
  converges we obtain an efficient approximate sampling algorithm for
  the measure $\bar\nu_{\dis}$ (respectively, $\bar\nu_{\ord}$)
  induced by the disordered polymer model
  by~\cite[Theorem~10]{HelmuthAlgorithmic2}; note that we have
  verified the conditions of this theorem in the previous section. By
  Lemma~\ref{lem:DisQual}, we thus obtain an efficient approximate
  sampling algorithm for $\mu_{\dis}$, the random cluster model
  conditional on the event that the configuration lies in
  $\Omega_{\dis}$, when $\beta\le \beta_{1}$.  Similarly we obtain
  efficient approximate sampling algorithms for $\mu_{\ord}$ when
  $\beta\geq \beta_{0}$.  By the approximate counting part of
  Theorem~\ref{thmAlgorithm}, which we have already proved, we can
  efficiently approximate the relative probabilities of
  $\Omega_{\ord}$ and $\Omega_{\dis}$. We thus obtain an efficient
  approximate sampling algorithm for the $q$-random cluster model by
  Corollary~\ref{cor:polyapprox}.
\end{proof}

\subsection{Application to random $\Delta$-regular  graphs}
\label{sec:RRG}

Corollary~\ref{thmAlgRandom} follows directly from Theorem~\ref{thmAlgorithm} and Proposition~\ref{prop:RRG-exp} in Section~\ref{sec:RRGexp}.

\section*{Acknowledgements}
We thank Anand Louis for pointing us to~\cite{guruswami2013rounding} and Guus Regts for helpful comments.
WP supported in part by NSF grants DMS-1847451 and CCF-1934915. TH was
supported in part by EPSRC grant EP/P003656/1.


\begin{thebibliography}{10}

\bibitem{achlioptas2008algorithmic}
D.~Achlioptas and A.~Coja-Oghlan.
\newblock Algorithmic barriers from phase transitions.
\newblock In {\em 2008 49th Annual IEEE Symposium on Foundations of Computer
  Science}, pages 793--802. IEEE, 2008.

\bibitem{barvinok2019weighted}
A.~Barvinok and G.~Regts.
\newblock Weighted counting of solutions to sparse systems of equations.
\newblock {\em Combinatorics, Probability and Computing}, 28(5):696--719, 2019.

\bibitem{blanca2020sampling}
A.~Blanca, A.~Galanis, L.~A. Goldberg, D.~{\v{S}}tefankovi{\v{c}}, E.~Vigoda,
  and K.~Yang.
\newblock Sampling in uniqueness from the {P}otts and random-cluster models on
  random regular graphs.
\newblock {\em SIAM Journal on Discrete Mathematics}, 34(1):742--793, 2020.

\bibitem{blanca2021random}
A.~Blanca and R.~Gheissari.
\newblock Random-cluster dynamics on random regular graphs in tree uniqueness.
\newblock {\em Communications in Mathematical Physics}, pages 1--45, 2021.

\bibitem{blanca2015dynamics}
A.~Blanca and A.~Sinclair.
\newblock Dynamics for the mean-field random-cluster model.
\newblock In {\em 18th International Workshop on Approximation Algorithms for
  Combinatorial Optimization Problems, APPROX 2015, and 19th International
  Workshop on Randomization and Computation, RANDOM 2015}, pages 528--543.
  Schloss Dagstuhl-Leibniz-Zentrum fur Informatik GmbH, Dagstuhl Publishing,
  2015.

\bibitem{bollobas1980probabilistic}
B.~Bollob{\'a}s.
\newblock A probabilistic proof of an asymptotic formula for the number of
  labelled regular graphs.
\newblock {\em European Journal of Combinatorics}, 1(4):311--316, 1980.

\bibitem{bollobas1988isoperimetric}
B.~Bollob{\'a}s.
\newblock The isoperimetric number of random regular graphs.
\newblock {\em European Journal of Combinatorics}, 9(3):241--244, 1988.

\bibitem{bordewich2016mixing}
M.~Bordewich, C.~Greenhill, and V.~Patel.
\newblock Mixing of the {G}lauber dynamics for the ferromagnetic {P}otts model.
\newblock {\em Random Structures \& Algorithms}, 48(1):21--52, 2016.

\bibitem{PottsAll2019stoc}
C.~Borgs, J.~Chayes, T.~Helmuth, W.~Perkins, and P.~Tetali.
\newblock Efficient sampling and counting algorithms for the {P}otts model on
  $\mathbb{Z}^d$ at all temperatures.
\newblock In {\em Proceedings of the 52nd Annual ACM SIGACT Symposium on Theory
  of Computing}, STOC 2020, page 738–751, New York, NY, USA, 2020.
  Association for Computing Machinery.

\bibitem{BorgsChayesKahnLovasz}
C.~Borgs, J.~Chayes, J.~Kahn, and L.~Lov{\'a}sz.
\newblock Left and right convergence of graphs with bounded degree.
\newblock {\em Random Structures \& Algorithms}, 42(1):1--28, 2013.

\bibitem{borgs2012tight}
C.~Borgs, J.~T. Chayes, and P.~Tetali.
\newblock Tight bounds for mixing of the {S}wendsen--{W}ang algorithm at the
  {P}otts transition point.
\newblock {\em Probability Theory and Related Fields}, 152(3-4):509--557, 2012.

\bibitem{borgs1991finite}
C.~Borgs, R.~Koteck{\'y}, and S.~Miracle-Sol{\'e}.
\newblock Finite-size scaling for {P}otts models.
\newblock {\em Journal of Statistical Physics}, 62(3-4):529--551, 1991.

\bibitem{cannon2019bipartite}
S.~Cannon and W.~Perkins.
\newblock Counting independent sets in unbalanced bipartite graphs.
\newblock In {\em Proceedings of the Fourteenth Annual ACM-SIAM Symposium on
  Discrete Algorithms (SODA)}, pages 1456--1466. SIAM, 2020.

\bibitem{carlson2020efficient}
C.~Carlson, E.~Davies, and A.~Kolla.
\newblock Efficient algorithms for the {P}otts model on small-set expanders.
\newblock {\em arXiv preprint arXiv:2003.01154}, 2020.

\bibitem{chayes1998graphical}
L.~Chayes and J.~Machta.
\newblock Graphical representations and cluster algorithms {II}.
\newblock {\em Physica A: Statistical Mechanics and its Applications},
  254(3-4):477--516, 1998.

\bibitem{chen2021sampling}
Z.~Chen, A.~Galanis, D.~{\v{S}}tefankovi{\v{c}}, and E.~Vigoda.
\newblock Sampling colorings and independent sets of random regular bipartite
  graphs in the non-uniqueness region.
\newblock {\em arXiv preprint arXiv:2105.01784}, 2021.

\bibitem{chertkov2006loop}
M.~Chertkov and V.~Y. Chernyak.
\newblock Loop series for discrete statistical models on graphs.
\newblock {\em Journal of Statistical Mechanics: Theory and Experiment},
  2006(06):P06009, 2006.

\bibitem{coja2018charting}
A.~Coja-Oghlan, C.~Efthymiou, N.~Jaafari, M.~Kang, and T.~Kapetanopoulos.
\newblock Charting the replica symmetric phase.
\newblock {\em Communications in Mathematical Physics}, 359(2):603--698, 2018.

\bibitem{coja2020replica}
A.~Coja-Oghlan, T.~Kapetanopoulos, and N.~M{\"u}ller.
\newblock The replica symmetric phase of random constraint satisfaction
  problems.
\newblock {\em Combinatorics, Probability and Computing}, 29(3):346--422, 2020.

\bibitem{CDKPR20}
M.~Coulson, E.~Davies, A.~Kolla, V.~Patel, and G.~Regts.
\newblock Statistical physics approaches to {U}nique {G}ames.
\newblock In {\em 35th Computational Complexity Conference (CCC 2020)}, 2020.

\bibitem{davies2021approximately}
E.~Davies and W.~Perkins.
\newblock {Approximately Counting Independent Sets of a Given Size in
  Bounded-Degree Graphs}.
\newblock In {\em 48th International Colloquium on Automata, Languages, and
  Programming (ICALP 2021)}, volume 198, pages 62:1--62:18, 2021.

\bibitem{dembo2010ising}
A.~Dembo and A.~Montanari.
\newblock Ising models on locally tree-like graphs.
\newblock {\em The Annals of Applied Probability}, 20(2):565--592, 2010.

\bibitem{dembo2014replica}
A.~Dembo, A.~Montanari, A.~Sly, and N.~Sun.
\newblock The replica symmetric solution for {P}otts models on d-regular
  graphs.
\newblock {\em Communications in Mathematical Physics}, 327(2):551--575, 2014.

\bibitem{duminilcopin}
H.~Duminil-Copin.
\newblock Lectures on the {I}sing and {P}otts models on the hypercubic lattice.
\newblock In {\em PIMS-CRM Summer School in Probability}, pages 35--161.
  Springer, 2017.

\bibitem{dyer2004relative}
M.~Dyer, L.~A. Goldberg, C.~Greenhill, and M.~Jerrum.
\newblock The relative complexity of approximate counting problems.
\newblock {\em Algorithmica}, 38(3):471--500, 2004.

\bibitem{edwards1988generalization}
R.~G. Edwards and A.~D. Sokal.
\newblock Generalization of the {F}ortuin-{K}asteleyn-{S}wendsen-{W}ang
  representation and {M}onte {C}arlo algorithm.
\newblock {\em Physical review D}, 38(6):2009, 1988.

\bibitem{galanis2011improved}
A.~Galanis, Q.~Ge, D.~{\v{S}}tefankovi{\v{c}}, E.~Vigoda, and L.~Yang.
\newblock Improved inapproximability results for counting independent sets in
  the hard-core model.
\newblock {\em Random Structures \& Algorithms}, 45(1):78--110, 2014.

\bibitem{galanis2020fast}
A.~Galanis, L.~A. Goldberg, and J.~Stewart.
\newblock Fast algorithms for general spin systems on bipartite expanders.
\newblock In {\em 45th International Symposium on Mathematical Foundations of
  Computer Science (MFCS 2020)}. Schloss Dagstuhl-Leibniz-Zentrum f{\"u}r
  Informatik, 2020.

\bibitem{galanis2021fast}
A.~Galanis, L.~A. Goldberg, and J.~Stewart.
\newblock Fast mixing via polymers for random graphs with unbounded degree.
\newblock {\em arXiv preprint arXiv:2105.00524}, 2021.

\bibitem{galanis2016inapproximability}
A.~Galanis, D.~{\v{S}}tefankovi{\v{c}}, and E.~Vigoda.
\newblock Inapproximability of the partition function for the antiferromagnetic
  {I}sing and hard-core models.
\newblock {\em Combinatorics, Probability and Computing}, 25(4):500--559, 2016.

\bibitem{galanis2019swendsen}
A.~Galanis, D.~{\v{S}}tefankovi{\v{c}}, and E.~Vigoda.
\newblock Swendsen-{W}ang algorithm on the mean-field {P}otts model.
\newblock {\em Random Structures \& Algorithms}, 54(1):82--147, 2019.

\bibitem{galanis2016ferromagnetic}
A.~Galanis, D.~{\v{S}}tefankovi{\v{c}}, E.~Vigoda, and L.~Yang.
\newblock Ferromagnetic {P}otts model: Refined \#-{BIS}-hardness and related
  results.
\newblock {\em SIAM Journal on Computing}, 45(6):2004--2065, 2016.

\bibitem{gheissari2018mixing}
R.~Gheissari and E.~Lubetzky.
\newblock Mixing times of critical two-dimensional {P}otts models.
\newblock {\em Communications on Pure and Applied Mathematics},
  71(5):994--1046, 2018.

\bibitem{gheissari2020exponentially}
R.~Gheissari, E.~Lubetzky, and Y.~Peres.
\newblock Exponentially slow mixing in the mean-field {S}wendsen--{W}ang
  dynamics.
\newblock {\em Annales de l'Institut Henri Poincar{\'e}, Probabilit{\'e}s et
  Statistiques}, 56(1):68--86, 2020.

\bibitem{giardina2015quenched}
C.~Giardin{\`a}, C.~Giberti, R.~van~der Hofstad, and M.~L. Prioriello.
\newblock Quenched central limit theorems for the {I}sing model on random
  graphs.
\newblock {\em Journal of Statistical Physics}, 160(6):1623--1657, 2015.

\bibitem{goldberg2012approximating}
L.~A. Goldberg and M.~Jerrum.
\newblock Approximating the partition function of the ferromagnetic {P}otts
  model.
\newblock {\em Journal of the ACM (JACM)}, 59(5):25, 2012.

\bibitem{gore1999swendsen}
V.~K. Gore and M.~R. Jerrum.
\newblock The {S}wendsen--{W}ang process does not always mix rapidly.
\newblock {\em Journal of Statistical Physics}, 97(1):67--86, 1999.

\bibitem{grimmett2006random}
G.~R. Grimmett.
\newblock {\em The Random-Cluster Model}.
\newblock Springer-Verlag, second edition, 2006.

\bibitem{gruber1971general}
C.~Gruber and H.~Kunz.
\newblock General properties of polymer systems.
\newblock {\em Communications in Mathematical Physics}, 22(2):133--161, 1971.

\bibitem{guo2018}
H.~Guo and M.~Jerrum.
\newblock Random cluster dynamics for the {I}sing model is rapidly mixing.
\newblock {\em Ann. Appl. Probab.}, 28(2):1292--1313, 04 2018.

\bibitem{guruswami2013rounding}
V.~Guruswami and A.~K. Sinop.
\newblock Rounding {L}asserre {SDP}s using column selection and spectrum-based
  approximation schemes for graph partitioning and quadratic {IP}s.
\newblock {\em arXiv preprint arXiv:1312.3024}, 2013.

\bibitem{haggstrom1996random}
O.~H{\"a}ggstr{\"o}m.
\newblock The random-cluster model on a homogeneous tree.
\newblock {\em Probability Theory and Related Fields}, 104(2):231--253, 1996.

\bibitem{HelmuthAlgorithmic2}
T.~Helmuth, W.~Perkins, and G.~Regts.
\newblock Algorithmic {P}irogov-{S}inai theory.
\newblock {\em Probability Theory and Related Fields}, 176:851--895, 2020.

\bibitem{hoory2006expander}
S.~Hoory, N.~Linial, and A.~Wigderson.
\newblock Expander graphs and their applications.
\newblock {\em Bulletin of the American Mathematical Society}, 43(4):439--561,
  2006.

\bibitem{huijben2021sampling}
J.~Huijben, V.~Patel, and G.~Regts.
\newblock Sampling from the low temperature {P}otts model through a {M}arkov
  chain on flows.
\newblock {\em arXiv preprint arXiv:2103.07360}, 2021.

\bibitem{JKP2}
M.~Jenssen, P.~Keevash, and W.~Perkins.
\newblock Algorithms for \#{BIS}-hard problems on expander graphs.
\newblock {\em SIAM Journal on Computing}, 49(4):681--710, 2020.

\bibitem{jenssen2019revisited}
M.~Jenssen and W.~Perkins.
\newblock Independent sets in the hypercube revisited.
\newblock {\em Journal of the London Mathematical Society}, 102(2):645--669,
  2020.

\bibitem{jerrum1989approximating}
M.~Jerrum and A.~Sinclair.
\newblock Approximating the permanent.
\newblock {\em SIAM Journal on Computing}, 18(6):1149--1178, 1989.

\bibitem{jerrum1993polynomial}
M.~Jerrum and A.~Sinclair.
\newblock Polynomial-time approximation algorithms for the {I}sing model.
\newblock {\em SIAM Journal on Computing}, 22(5):1087--1116, 1993.

\bibitem{kotecky1986cluster}
R.~Koteck\'{y} and D.~Preiss.
\newblock Cluster expansion for abstract polymer models.
\newblock {\em Communications in Mathematical Physics}, 103(3):491--498, 1986.

\bibitem{krzakala2007gibbs}
F.~Krzaka{\l}a, A.~Montanari, F.~Ricci-Tersenghi, G.~Semerjian, and
  L.~Zdeborov{\'a}.
\newblock Gibbs states and the set of solutions of random constraint
  satisfaction problems.
\newblock {\em Proceedings of the National Academy of Sciences},
  104(25):10318--10323, 2007.

\bibitem{laanait1991interfaces}
L.~Laanait, A.~Messager, S.~Miracle-Sol{\'e}, J.~Ruiz, and S.~Shlosman.
\newblock Interfaces in the {P}otts model {I}: {P}irogov-{S}inai theory of the
  {F}ortuin-{K}asteleyn representation.
\newblock {\em Communications in Mathematical Physics}, 140(1):81--91, 1991.

\bibitem{liao2019counting}
C.~Liao, J.~Lin, P.~Lu, and Z.~Mao.
\newblock Counting independent sets and colorings on random regular bipartite
  graphs.
\newblock In {\em Approximation, Randomization, and Combinatorial Optimization.
  Algorithms and Techniques (APPROX/RANDOM 2019)}. Schloss
  Dagstuhl-Leibniz-Zentrum fuer Informatik, 2019.

\bibitem{lubetzky2010cutoff}
E.~Lubetzky and A.~Sly.
\newblock Cutoff phenomena for random walks on random regular graphs.
\newblock {\em Duke Mathematical Journal}, 153(3):475--510, 2010.

\bibitem{lucibello2014finite}
C.~Lucibello, F.~Morone, G.~Parisi, F.~Ricci-Tersenghi, and T.~Rizzo.
\newblock Finite-size corrections to disordered {I}sing models on random
  regular graphs.
\newblock {\em Physical Review E}, 90(1):012146, 2014.

\bibitem{mckay2004short}
B.~D. McKay, N.~C. Wormald, and B.~Wysocka.
\newblock Short cycles in random regular graphs.
\newblock {\em The Electronic Journal of Combinatorics}, pages R66--R66, 2004.

\bibitem{mezard2009information}
M.~Mezard and A.~Montanari.
\newblock {\em Information, physics, and computation}.
\newblock Oxford University Press, 2009.

\bibitem{molloy2018freezing}
M.~Molloy.
\newblock The freezing threshold for k-colourings of a random graph.
\newblock {\em Journal of the ACM (JACM)}, 65(2):1--62, 2018.

\bibitem{montanari2012weak}
A.~Montanari, E.~Mossel, and A.~Sly.
\newblock The weak limit of {I}sing models on locally tree-like graphs.
\newblock {\em Probability Theory and Related Fields}, 152(1-2):31--51, 2012.

\bibitem{montanari2005compute}
A.~Montanari and T.~Rizzo.
\newblock How to compute loop corrections to the {B}ethe approximation.
\newblock {\em Journal of Statistical Mechanics: Theory and Experiment},
  2005(10):P10011, 2005.

\bibitem{montenegro2006mathematical}
R.~R. Montenegro and P.~Tetali.
\newblock {\em Mathematical aspects of mixing times in Markov chains}.
\newblock Now Publishers Inc, 2006.

\bibitem{pirogov1975phase}
S.~A. Pirogov and Y.~G. Sinai.
\newblock Phase diagrams of classical lattice systems.
\newblock {\em Theoretical and Mathematical Physics}, 25(3):1185--1192, 1975.

\bibitem{rassmann2019number}
F.~Rassmann.
\newblock On the number of solutions in random graph k-colouring.
\newblock {\em Combinatorics, Probability and Computing}, 28(1):130--158, 2019.

\bibitem{sly2010computational}
A.~Sly.
\newblock Computational transition at the uniqueness threshold.
\newblock In {\em Proceedings of the Fifty-first Annual IEEE Symposium on
  Foundations of Computer Science}, FOCS 2010, pages 287--296. IEEE, 2010.

\bibitem{sly2014counting}
A.~Sly and N.~Sun.
\newblock Counting in two-spin models on d-regular graphs.
\newblock {\em The Annals of Probability}, 42(6):2383--2416, 2014.

\bibitem{swendsen1987nonuniversal}
R.~H. Swendsen and J.-S. Wang.
\newblock Nonuniversal critical dynamics in {M}onte {C}arlo simulations.
\newblock {\em Physical review letters}, 58(2):86, 1987.

\bibitem{trevisan2016lecture}
L.~Trevisan.
\newblock Lecture notes on graph partitioning, expanders and spectral methods.
\newblock \url{https://lucatrevisan.github.io/books/expanders-2016.pdf}, 2016.

\bibitem{weitz2006counting}
D.~Weitz.
\newblock Counting independent sets up to the tree threshold.
\newblock In {\em Proceedings of the Thirty-Eighth Annual ACM {S}ymposium on
  Theory of {C}omputing}, STOC 2006, pages 140--149. ACM, 2006.

\bibitem{wormald1981asymptotic}
N.~C. Wormald.
\newblock The asymptotic distribution of short cycles in random regular graphs.
\newblock {\em Journal of Combinatorial Theory, Series B}, 31(2):168--182,
  1981.

\end{thebibliography}
\end{document}